\documentclass[10pt]{amsart}
\usepackage{amssymb,amsfonts,lineno,amsthm,amscd,stmaryrd}
\usepackage[leqno]{amsmath}

\newtheorem{thm}{Theorem}
\newtheorem{cor}[thm]{Corollary}
\newtheorem{lem}[thm]{Lemma}
\newtheorem{prop}[thm]{Proposition}
\newtheorem{conj}[thm]{Conjecture}
\theoremstyle{definition}
\newtheorem{definition}[thm]{Definition}
\newtheorem{remark}[thm]{Remark}
\newtheorem{example}[thm]{Example}
\newtheorem{openproblem}[thm]{Open Problem}

\numberwithin{thm}{section}
\numberwithin{equation}{section}

\newcommand{\EQ}[1]{\eqref{eq:#1}} 
\newcommand{\LEM}[1]{Lemma~\ref{lem:#1}}    
\newcommand{\DEF}[1]{Definition~\ref{def:#1}}    
\newcommand{\THM}[1]{Theorem~\ref{thm:#1}}  
\newcommand{\REM}[1]{Remark~\ref{rem:#1}}  
\newcommand{\PROP}[1]{Proposition~\ref{prop:#1}}  
\newcommand{\COR}[1]{Corollary~\ref{cor:#1}} 
\newcommand{\SEC}[1]{Section~\ref{sec:#1}}  

\newcommand{\OPROB}[1]{Open Problem~\ref{oprob:#1}}  
\newcommand{\CONJ}[1]{Conjecture~\ref{conj:#1}}  

\DeclareMathOperator{\diam}{diam}
\DeclareMathOperator{\dist}{dist}
\DeclareMathOperator{\USC}{USC}
\DeclareMathOperator{\LSC}{LSC}

\DeclareMathOperator{\Lip}{Lip}

\newcommand{\R}{\ensuremath{\mathbb{R}}}
\newcommand{\iden}{\ensuremath{I_n}}
\newcommand{\ep}{\varepsilon}

\newcommand{\distep}[3]{\rho_{#1}(#2,#3)}
\newcommand{\path}{d}
\newcommand{\ILP}{\Delta_\infty^+}
\newcommand{\ILM}{\Delta_\infty^-}
\newcommand{\IL}{\Delta_\infty}
\newcommand{\E}{\mathbb{E}}
\newcommand{\prob}{\mathbb{P}}

\begin{document}
\title[A finite difference approach to the infinity {L}aplace equation]{A finite difference approach to the infinity {L}aplace equation and tug-of-war games}
\author{Scott N. Armstrong}
\address{Department of Mathematics,
University of California, Berkeley, CA 94720.}
\email{sarm@math.berkeley.edu}
\author{Charles K. Smart}
\address{Department of Mathematics,
University of California, Berkeley, CA 94720.}
\email{smart@math.berkeley.edu}

\date{\today}


\begin{abstract}
We present a modified version of the two-player ``tug-of-war" game introduced by Peres, Schramm, Sheffield, and Wilson \cite{Peres:2009}. This new tug-of-war game is identical to the original except near the boundary of the domain $\partial \Omega$, but its associated value functions are more regular. The dynamic programming principle implies that the value functions satisfy a certain finite difference equation. By studying this difference equation directly and adapting techniques from viscosity solution theory, we prove a number of new results.

We show that the finite difference equation has unique maximal and minimal solutions, which are identified as the value functions for the two tug-of-war players. We demonstrate uniqueness, and hence the existence of a value for the game, in the case that the running payoff function is nonnegative. We also show that uniqueness holds in certain cases for sign-changing running payoff functions which are sufficiently small. In the limit $\ep \to 0$, we obtain the convergence of the value functions to a viscosity solution of the normalized infinity Laplace equation. 

We also obtain several new results for the normalized infinity Laplace equation $-\Delta_\infty u = f$. In particular, we demonstrate the existence of solutions to the Dirichlet problem for any bounded continuous $f$, and continuous boundary data, as well as the uniqueness of solutions to this problem in the generic case. We present a new elementary proof of uniqueness in the case that $f>0$, $f< 0$, or $f\equiv 0$. The stability of the solutions with respect to $f$ is also studied, and an explicit continuous dependence estimate from $f\equiv 0$ is obtained.
\end{abstract}

\maketitle


\section{Introduction}

In this article, we use a finite difference approximation to study the normalized infinity Laplace partial differential equation
\begin{equation}\label{eq:infinity-Laplace-PDE}
-\Delta_\infty u := -|Du|^{-2} \sum_{i,j=1}^n u_{x_ix_j}u_{x_i} u_{x_j} = f
\end{equation}
in a bounded domain $\Omega \subseteq \R^n$. The equation \EQ{infinity-Laplace-PDE} arises in the $L^\infty$ calculus of variations as the Euler-Lagrange equation for properly interpreted minimizers of the energy functional $u \mapsto \| D u \|_{L^\infty(\Omega)}$. A viscosity solution $u$ of \EQ{infinity-Laplace-PDE} with $f\equiv 0$ is a so-called \emph{absolutely minimizing Lipschitz} function. This means that for every open subset $V \subseteq \Omega$, the function $u |_{\bar V}$ has the smallest possible Lipschitz constant in the class of all Lipschitz functions on $\bar V$ which are equal to $u$ on $\partial V$.  We refer to Aronsson, Crandall, and Juutinen \cite{Aronsson:2004} for more background and details.

Recently, Peres, Schramm, Sheffield, and Wilson \cite{Peres:2009} showed that equation \EQ{infinity-Laplace-PDE} also arises in the study of certain two-player, zero-sum, stochastic games. They introduced a random-turn game called \emph{$\ep$-tug-of-war}, in which two players try to move a token in an open set $\Omega$ toward a favorable spot on the boundary $\partial \Omega$. During each round of the game, a fair coin is tossed to determine which player may move the token, and the winner may move it a maximum distance of $\ep > 0$ from its previous position. The payoff is determined by a running payoff function $f$, and a terminal payoff function $g$. We describe tug-of-war in more detail in \SEC{two}.

In \cite{Peres:2009} it was shown that under the hypothesis that the running payoff function $f$ is positive, negative, or identically zero, this game has an expected value. Moreover, they showed that as $\ep \to 0$, the expected value function converges to the unique viscosity solution $u$ of the equation $-\Delta_\infty u = f$ with $u=g$ on $\partial \Omega$. The probabilistic methods employed in \cite{Peres:2009} yielded new results and a better understanding of the PDE \EQ{infinity-Laplace-PDE}. Connections between stochastic games and the infinity Laplace equation have also been investigated by Barron, Evans, and Jensen \cite{Barron:2008}.

In this paper, we use PDE techniques to study the value functions for tug-of-war games and the solutions of equation \EQ{infinity-Laplace-PDE}. By changing the rules of tug-of-war when the token is very near the boundary, we obtain a new game whose value functions better approximate solutions of \EQ{infinity-Laplace-PDE}. In particular, the upper and lower value functions of this modified game, which we call \emph{boundary-biased $\ep$-step tug-of-war}, are continuous. In fact, if $f\equiv 0$ and $g$ is Lipschitz, then the (unique) value function is a \emph{minimizing Lipschitz extension} of $g$ to $\bar\Omega$, for each $\ep > 0$. Furthermore, the upper and lower value functions are equal, and hence the game possesses a value, under more general hypotheses than is presently known for standard $\ep$-tug-of-war.

In contrast to \cite{Peres:2009}, we make little use of probabilistic methods in this paper. Instead, we study the difference equation
\begin{equation}\label{eq:intro-FD}
-\Delta_\infty^\ep u(x) = f(x) \quad \mbox{in} \  \Omega,
\end{equation}
which is derived from the dynamic programming principle. The finite difference operator $-\Delta_\infty^\ep$ is defined in \EQ{finite-difference-infinity-laplacian}, below. We show that if $f$ is continuous, bounded, and does not change sign in $\Omega$, then \EQ{intro-FD} possesses a unique solution subject to any given, continuous boundary data $g$. It follows that the boundary-biased game has a value in this case. Furthermore, we show that for each bounded, continuous $f$, any sequence of solutions of \EQ{intro-FD} converges (after possibly taking a subsequence) to a solution of the continuum equation \EQ{infinity-Laplace-PDE} as $\ep \to 0$. In the case $f\equiv 0$, Oberman \cite{Oberman:2005} and Le Gruyer \cite{LeGruyer:2007} obtained similar existence, uniqueness and convergence results for a similar difference equation on a finite graph.

Our analysis of \EQ{intro-FD} yields several new results concerning the continuum equation. For any $f \in C(\Omega) \cap L^\infty(\Omega)$ and $g\in C(\partial \Omega)$, we show that the Dirichlet problem
\begin{equation} \label{eq:intro-DP}
\left\{ \begin{array}{lll}
-\Delta_\infty u = f & \mbox{ in } & \Omega, \\
u = g & \mbox{ on } & \partial \Omega,
\end{array}\right.
\end{equation}
possesses a unique maximal and minimal viscosity solution. Existence has been previously shown only for $f$ satisfying $f > 0$, $f < 0$, or $f\equiv 0$, in which case we also have uniqueness. The latter uniqueness result appears in \cite{Peres:2009} as well as the paper of Lu and Wang \cite{Lu:2008}. The case $f\equiv 0$ is Jensen's famous result \cite{Jensen:1993}, and other proofs in this case have appeared in \cite{Barles:2001,Aronsson:2004,Crandall:2007}.

Here we give a new, elementary proof of uniqueness under the assumption that $f>0$, $f< 0$, or $f\equiv 0$. Unlike previous proofs, our argument does not use deep viscosity solution theory or probabilistic techniques. Instead, our proof is based on the simple observation that by modifying a solution of the PDE \EQ{infinity-Laplace-PDE} by ``maxing over $\ep$-balls," we obtain a subsolution of the finite difference equation \EQ{intro-FD}, possibly with small error.

It is known that there may exist multiple viscosity solutions of the boundary value problem \EQ{intro-DP} in the case that $f$ changes sign (see \cite[Section 5.4]{Peres:2009}). However, in this article we demonstrate that uniqueness holds for \emph{generic} $f$. That is, for fixed $\tilde{f} \in C(\Omega) \cap L^\infty(\Omega)$ and $g \in \partial \Omega$, we show that the problem \EQ{intro-DP} has a unique solution for $f = \tilde{f} + c$ for all but at most countably many $c\in \R$. See \THM{uniqueness-generic} below. This result provides an affirmative answer to a question posed in \cite{Peres:2009}.

Other new theorems obtained in this paper include a result regarding the stability of solutions of \EQ{intro-DP} and an explicit continuous dependence estimate from $f\equiv 0$.

In \SEC{two}, we review our notation and definitions, describe the tug-of-war games in more detail, state our main results, and give an outline of the rest of the paper.


\section{Preliminaries and main results}\label{sec:two}

\subsection*{Notation}

Throughout this article, $\Omega$ denotes a bounded, connected, and open subset of $\R^n$, $f$ denotes an element of $C(\Omega)\cap L^\infty(\Omega)$, $g$ denotes an element of $C(\partial \Omega)$, and $\ep$ denotes a small number satisfying $0 < \ep < 1$. At various points we impose additional hypotheses on $\Omega$, $f$, $g$, and $\ep$.

If $x,y \in \R^n$, we denote the usual Euclidean inner product by $\langle x, y \rangle$, and use $|x|$ to denote the Euclidean length of $x$. If $E \subseteq \R^n$, we denote the closure of $E$ by $\bar E$.  The set of upper semicontinuous functions on $\bar\Omega$ is denoted by $\USC(\bar\Omega)$, and likewise the set of lower semicontinuous functions is denoted by $\LSC(\bar\Omega)$. We denote the $n$-by-$n$ identity matrix by $\iden$. If $x\in \R^n$, then $x \otimes x$ denotes the matrix $(x_ix_j)$.

We denote path distance on $\bar \Omega$ by $\path$. That is, $\path(x,y)$ is the infimum of the lengths of all possible Lipschitz paths $\gamma : [0,1] \to \bar \Omega$ with $\gamma(0)= x$ and $\gamma(1)=y$. If $x\in \bar\Omega$ and $E \subseteq \bar \Omega$, we set $\dist(x,E) : = \inf\{ \path(x,y) : y \in E \}$ and define the set $\Omega_\delta := \left\{ x\in \Omega : \dist(x,\partial \Omega) > \delta \right\}$. The open ball with respect path distance with center $x_0 \in \Omega$ and radius $r > 0$ is denoted by
\begin{equation*}
\Omega(x_0, r) : = \left\{ x \in \Omega : d(x,x_0) < r \right\}.
\end{equation*}
We denote by $\diam(\Omega)$ the diameter of $\Omega$ with respect to $d$; i.e.,
\begin{equation*}
\diam(\Omega) = \sup\{ d(x,y) : x,y\in \bar \Omega \}.
\end{equation*}
We require that the boundary $\partial \Omega$ of $\Omega$ is sufficiently regular so that 
\begin{equation*}
\diam(\Omega) <  \infty.
\end{equation*}
The open ball with respect to Euclidean distance is denoted $B(x_0, r)$. We usually refer to $B(x_0, r)$ only when $x_0 \in \Omega_r$, in which case $B(x_0, r) = \Omega(x_0, r)$. It is somewhat inconvenient to work with path distance, but it is needed in \SEC{three} to handle difficulties which appear near the boundary of the domain.

If $K$ is a compact subset of $\bar \Omega$ and $h:K \to \R$ is continuous, we define the \emph{modulus} of $h$ on $K$ by
\begin{equation*}
\omega_h(s) : = \max \left\{ t |h(x) - h(y) | \, : \, t |x-y| \leq s \ \mbox{and} \ t \leq 1 \right\}, \quad s\geq 0.
\end{equation*}
It is easy to check that $|h(x) - h(y)| \leq \omega_h(|x-y|) \leq  \omega_h(\path(x,y))$, and that $\omega_h$ is continuous, nondecreasing, and concave on $[0,\infty)$, and $\omega_h(0) = 0$. In particular,
\begin{equation}\label{eq:omega-g-concave}
\omega_h(ts) \leq t \omega_h(s) \quad \mbox{for every} \ s \geq 0, t\geq 1.
\end{equation}
We call any function $\omega$ with the properties above a \emph{modulus of continuity for} $h$.

\subsection*{The infinity Laplace equation}

We recall the notion of a viscosity solution of the (normalized) infinity Laplace equation. For a $C^2$ function $\varphi$ defined in a neighborhood of $x \in \R^n$, we define the operators
\begin{equation*}
\ILP\varphi(x) : = \begin{cases}
|D\varphi(x)|^{-2} \left\langle  D^2\varphi(x) \cdot D\varphi(x) , D\varphi(x) \right\rangle & \mbox{if} \ D\varphi(x) \neq 0, \\
\max\left\{ \left\langle D^2\varphi(x)v, v\right\rangle : |v| =1 \right\} & \mbox{otherwise},
\end{cases}
\end{equation*}
and
\begin{equation*}
\ILM\varphi(x) : = \begin{cases}
|D\varphi(x)|^{-2} \left\langle  D^2\varphi(x) \cdot D\varphi(x) , D\varphi(x) \right\rangle & \mbox{if} \ D\varphi(x) \neq 0, \\
\min\left\{ \left\langle D^2\varphi(x)v, v\right\rangle : |v| =1 \right\} & \mbox{otherwise}.
\end{cases}
\end{equation*}
Notice that $\ILM\varphi(x) := - \ILP (-\varphi)(x)$. For any $\varphi \in C^2(\Omega)$, the map $x \mapsto \ILP \varphi (x)$ is upper semicontinuous in $\Omega$, while $x \mapsto \ILM\varphi(x)$ is lower semicontinuous, and the two are equal (and hence continuous) on the set $\{ x \in \Omega : D \varphi(x) \neq 0 \}$.

\begin{definition}\label{def:viscosity}
An upper semicontinuous function $u \in \USC(\Omega)$ is a \emph{viscosity subsolution} of the normalized infinity Laplace equation
\begin{equation}\label{eq:defviscsol}
-\Delta_\infty u = f \quad \mbox{in} \ \Omega
\end{equation}
if, for every polynomial $\varphi$ of degree 2 and $x_0\in \Omega$ such that
\begin{equation*}
x\mapsto u(x) - \varphi(x) \quad \mbox{has a strict local maximum at} \ x=x_0,
\end{equation*} 
we have
\begin{equation}\label{eq:defviscsol-sub}
- \ILP \varphi(x_0) \leq f(x_0).
\end{equation}
Likewise, a lower semicontinuous function $u \in \LSC(\Omega)$ is a \emph{viscosity supersolution} of \EQ{defviscsol} if, for every polynomial $\varphi$ of degree 2 and $x_0\in \Omega$ such that
\begin{equation*}
x\mapsto v(x) - \varphi(x) \quad \mbox{has a strict local minimum at} \ x=x_0,
\end{equation*} 
we have
\begin{equation}\label{eq:defviscsol-super}
- \ILM \varphi(x_0) \geq f(x_0).
\end{equation}
We say that $u$ is a \emph{viscosity solution} of \EQ{defviscsol} if $u$ is both a viscosity subsolution and viscosity supersolution of \EQ{defviscsol}.

If we strengthen our definitions of viscosity subsolution/supersolution by requiring \EQ{defviscsol-sub}/\EQ{defviscsol-super} to hold whenever $\varphi \in C^2$ and $u-\varphi$ has a (possibly not strict) local maximum/minimum at $x_0$, then we obtain equivalent definitions.

If $u \in C(\Omega)$ is a viscosity subsolution (supersolution) of \EQ{defviscsol}, then we often write
\begin{equation} \label{eq:defviscsol-pedantic}
-\Delta_\infty u \leq (\geq) f \quad \mbox{in} \ \Omega.
\end{equation}
We emphasize that the differential inequality \EQ{defviscsol-pedantic} is to be understood only in the viscosity sense.

In the case that $f\equiv 0$, subsolutions, supersolutions, and solutions of \EQ{defviscsol} are called \emph{infinity subharmonic}, \emph{infinity superharmonic}, and \emph{infinity harmonic}, respectively.
\end{definition}

\begin{remark}
In this paper, the symbol $-\Delta_\infty$ always denotes the \emph{normalized} or \emph{$1$-homogeneous} infinity Laplacian operator
\begin{equation}\label{eq:normalized-infinity-laplacian}
-\Delta_\infty u := - |Du|^{-2} \left\langle D^2u \cdot Du, Du \right\rangle.
\end{equation}
In the PDE literature, it is more customary to reserve $-\Delta_\infty$ to denote the operator $- \left\langle D^2u \cdot Du, Du \right\rangle$. We break from this convention since the normalized infinity Laplacian operator is more natural from the perspective of tug-of-war games, and is therefore the focus of this article. We also point out that there is no difference between the two resulting equations (in the viscosity sense) when the right-hand side $f\equiv 0$. We henceforth drop the modifier \emph{normalized} and refer to \EQ{normalized-infinity-laplacian} as the \emph{infinity Laplacian} and the equation $-\Delta_\infty u = f$ as the \emph{infinity Laplace equation}.
\end{remark}

\subsection*{Tug-of-war}

Let us briefly review the notion of two-player, zero-sum, random-turn tug-of-war games, which were first introduced by Peres, Schramm, Sheffield, and Wilson \cite{Peres:2009}. Fix a number $\ep > 0$. The dynamics of the game are as follows. A token is placed at an initial position $x_0 \in \Omega$. At the $k$th stage of the game, Player I and Player II select points $x^I_k$ and $x^{II}_k$, respectively, each belonging to a specified set $A(x_{k-1},\ep) \subseteq \bar\Omega$. The game token is then moved to $x_k$, where $x_k$ is chosen randomly so that $x_k = x^I_k$ with probability $P=P(x_{k-1},x^I_k,x^{II}_k)$ and $x_k = x^{II}_k$ with probability $1-P$, where $P$ is a given function. After the $k$th stage of the game, if $x_k \in \Omega$, then the game continues to stage $k+1$. Otherwise, if $x_{k} \in \partial \Omega$, the game ends and Player II pays Player I the amount
\begin{equation}\label{eq:game-payoff}
\mathrm{Payoff} = g(x_{k}) + \frac{\ep}{2} \sum_{j=1}^k q(\ep,x_{j-1},x_j) f(x_{j-1}),
\end{equation} 
where $q$ is a given function. We call $g$ the \emph{terminal payoff function} and $f$ the \emph{running payoff function}. Of course, Player I attempts to maximize the payoff, while Player II attempts to minimize it.

A \emph{strategy} for a Player I is a mapping $\sigma_I$ from the set of all possible partially played games $(x_0,x_1,\ldots,x_{k-1})$ to moves $x^I_k \in A(x_{k-1},\ep)$, and a strategy for Player II is defined in the same way.

Given a strategy $\sigma_I$ for Player I and a strategy $\sigma_{II}$ for Player II, we denote by $F_I(\sigma_I,\sigma_{II})$ and $F_{II}(\sigma_I,\sigma_{II})$ the expected value of the expression \EQ{game-payoff} if the game terminates with probability one, and this expectation is defined in $[-\infty,\infty]$. Otherwise, we set $F_I(\sigma_I,\sigma_{II}) = - \infty$ and $F_{II}(\sigma_I,\sigma_{II}) = + \infty$. (If the players decide to play in a way that makes the probability of the game terminating less than 1, then we penalize both players an infinite amount.)

The \emph{value} of the game for Player I is the quantity $\sup_{\sigma_I} \inf_{\sigma_{II}} F_I(\sigma_I,\sigma_{II})$, where the supremum is taken over all possible strategies for Player I and the infimum over all possible strategies for Player II. It is the minimum amount that Player I should expect to win at the conclusion of the game. Similarly, the \emph{value} of the game for Player II is $\inf_{\sigma_{II}} \sup_{\sigma_I} F_{II}(\sigma_I,\sigma_{II})$, which is the maximum amount that Player II should expect to lose at the conclusion of the game. We denote the value for Player I as a function of the starting point $x \in \Omega$ by $V_I^\ep(x)$, and similarly the value for Player II by $V_{II}^\ep(x)$. We extend the value functions to $\partial \Omega$ by setting $V_I^\ep = V_{II}^\ep = g$ there. It is clear that $V_I^\ep \leq V_{II}^\ep$. The game is said to have a \emph{value} if $V_I^\ep \equiv V_{II}^\ep =:V^\ep$.

The tug-of-war game studied in \cite{Peres:2009}, which in this paper we call \emph{standard $\ep$-step tug-of-war}, is the game described above for\footnote{The game described in \cite{Peres:2009} actually requires the players to select points in the slightly smaller set $A(x,\ep) = \Omega(x,\ep) \cup ( \{ y\in \partial \Omega : d(x,y) < \ep \}$, when the current position of the token is $x$. Technical difficulties arise in some of the probabilistic arguments in \cite{Peres:2009} if the players are allowed to move to points in $\{ y \in \bar \Omega: d(x,y) = \ep\}$. This small difference does not concern us here.}
\begin{equation*}
A(x,\ep) = \bar \Omega(x,\ep), \quad P = \frac{1}{2}, \quad \mbox{and} \ \ q(\ep,x,y) = \ep.
\end{equation*}
In other words, the players must choose points in the $\ep$-ball centered at the current location of the token, a fair coin is tossed to determine where the token is placed, and Player II accumulates a debt to Player I which is increased by $\frac{1}{2}\ep^2 f(x_{k-1})$ after the $k$th stage.

\medskip

According to the dynamic programming principle, the value functions for Player I and Player II for standard $\ep$-turn tug-of-war satisfy the relation
\begin{equation}\label{eq:standard-dynamic-programming}
2V(x) - \left( \sup_{\bar\Omega(x,\ep)} V + \inf_{\bar\Omega(x,\ep)} V \right) = \ep^2 f(x),\quad x\in \Omega.
\end{equation}
If we divide the left side of \EQ{standard-dynamic-programming} by $\ep^2$, we have a good approximation to the negative of the second derivative of $u$ in the direction of $Du(x)$, provided that $u$ is smooth, $Du(x) \neq 0$, and $\dist (x,\partial \Omega) \geq \ep$ (indeed, see \LEM{discrete-to-continuous} below). Thus we might expect that in the limit as $\ep \to 0$, the value functions for Players I and II converge to a solution of the boundary-value problem
\begin{equation}\label{eq:Peres2009}
\left\{ \begin{aligned}
& - \Delta_\infty u = f & \mbox{in} & \ \Omega,\\
& u = g & \mbox{on} & \ \partial \Omega.
\end{aligned} \right.
\end{equation}
This is indeed the case for certain running payoff functions $f$, as was shown in \cite{Peres:2009}.

\begin{thm}[Peres, Shramm, Sheffield, Wilson \cite{Peres:2009}]\label{thm:Peres2009}
Assume that $f \in C(\bar\Omega)$ and
\begin{equation}\label{eq:Peres2009hyp}
f \equiv 0 \ \ \mbox{or} \ \  \min_{\bar \Omega} f > 0 \ \ \mbox{or} \ \ \max_{\bar \Omega} f < 0.
\end{equation}
Then the boundary value problem \EQ{Peres2009} has a unique viscosity solution $u \in C(\bar\Omega)$, and for every $\ep > 0$, standard $\ep$-step tug-of-war possesses a value $V^\ep$. Moreover, $V^\ep \rightarrow u$ uniformly in $\bar \Omega$ as $\ep \to 0$. 
\end{thm}

One of the goals for this present work is to develop PDE methods for tug-of-war games and the infinity Laplace equation. Some of the difficulties in the analysis of standard $\ep$-step tug-of-war are due to the discontinuity of the value functions. To understand this phenomena, we study the following simple example.

\begin{example}
Let $\ep = 1/k$ for some positive integer $k\geq 2$, and consider standard $\ep$-step tug-of-war played on the unit interval $\Omega = (0,1)$, with vanishing running payoff function, and terminal payoff function given by and $g(0) = 0$ and $g(1) = 1$. It is easy to see that the value function $V_k$ must be constant on the intervals $I_j := ( (j-1)/k, j / k )$ for each $j = 1,\ldots, k$. Denote its value on the inteval $I_j$ by $v_j$, and write $v = (v_1,\ldots,v_k)$. The dynamic programming relation \EQ{standard-dynamic-programming} now yields the system
\begin{equation}\label{eq:stupid-example}
\left\{ \begin{aligned}
0 & = 2v_1 - v_2, & \\
0 & = -v_j + 2v_{j+1} - v_{j+2} & (2 \leq j \leq k-1), \\
1 & = -v_{k-1} + 2v_k.
\end{aligned}\right.
\end{equation}
The system \EQ{stupid-example} has the unique solution $v_j = (j+1) / (k+1)$ for $j=1,\ldots k$. Thus the value function for this standard tug-of-war game is a step function which approximates the continuum solution, given by $V(x) = x$.

This one dimensional example can be lifted into higher dimensions by considering $\Omega = B(0,2) \setminus \bar B(0,1)$ and setting the terminal payoff function to $1$ on $\partial B(0,2)$ and to $0$ on $\partial B(0,1)$. It is clear that the value function for standard $1/k$-step tug-of-war is then $V_k(|x|-1)$ for $1 < |x| < 2$.

\end{example}

\subsection*{Boundary-biased tug-of-war}

In this article, we study a slight variant of standard $\ep$-turn tug-of-war, which we call \emph{boundary-biased $\ep$-step tug-of-war}. This is the game described in the previous section, where we set
\begin{equation*}
A(x,\ep) = \bar\Omega(x,\ep), \quad P(x,y,z) =  \frac{\rho_\ep(x,z)}{\rho_\ep(x,y)+\rho_\ep(x,z)}, \quad \mbox{and} \ \ q(\ep,x,y) = \rho_\ep(x,y),
\end{equation*}
where
\begin{equation*}
\distep{\ep}{x}{y} := \begin{cases}
\max\{ d(x,y), \ep \} & \mbox{if} \ x, y \in \Omega, \\
d(x,y) & \mbox{if} \ x \in \partial \Omega \ \mbox{or} \ y \in \partial \Omega.
\end{cases} 
\end{equation*}
The dynamics of the boundary-biased game and the accumulation of the running payoff are no different from that of the standard game while the token lies in the set $\Omega_\ep$, as $q(\ep,x_k,x_{k+1}) = \rho_\ep(x_k,x_k^I) = \rho^\ep(x_k,x^{II}_k) = \ep$ if $x_k \in \Omega_\ep$. The distinction between the games occurs near the boundary, where the boundary-biased game gives a player who wishes to terminate the game by exiting at a nearby boundary point a larger probability of winning the coin toss, if the other player selects a point in the domain $\Omega$ or a boundary point further away. The payoff has also been altered slightly from the standard game, so that small jumps to the boundary do not accrue as much running payoff.

Boundary-biased $\ep$-step tug-of-war is indeed only a slight variant of standard $\ep$-step tug-of-war. In fact, by combining results in this paper with those in \cite{Peres:2009}, we can show that the value functions for the two games differ by $O(\ep)$. 

Our purpose for considering boundary-biased tug-of-war in this article is precisely because the value functions are more regular, as we see below. In particular, they are continuous, and uniformly bounded and equicontinuous along sequences $\ep_j \downarrow 0$. These properties allow us to adapt techniques from viscosity solution theory.

The analogue of \EQ{standard-dynamic-programming} for boundary-biased $\ep$-step tug-of-war, derived from the dynamic programming principle, is the equation
\begin{equation}\label{eq:dynamic-prog-bb}
\sup_{y\in \bar\Omega(x,\ep)} \frac{V(x) - V(y)}{\rho_\ep(x,y)} - \sup_{y\in \bar\Omega(x,\ep)} \frac{V(y) - V(x)}{\rho_\ep(x,y)} = \ep f(x).
\end{equation}
Let us introduce the notation
\begin{equation*}
S^+_\ep u(x) := \sup_{y\in \bar\Omega(x,\ep)} \frac{u(y) - u(x)}{\rho_\ep(x,y)} \quad \mbox{and} \quad S^-_\ep u(x) := \sup_{y\in \bar\Omega(x,\ep)} \frac{u(x) - u(y)}{\rho_\ep(x,y)},
\end{equation*}
and
\begin{equation}\label{eq:finite-difference-infinity-laplacian}
\Delta_\infty^\ep u(x) := \frac{1}{\ep} \left( S^+_\ep u(x) - S^-_\ep u(x) \right).
\end{equation}
We may write \EQ{dynamic-prog-bb} as
\begin{equation}\label{eq:finite-difference-infinity-laplace}
-\Delta^\ep_\infty u = f \quad \mbox{in} \ \Omega.
\end{equation}
We call the operator $\Delta_\infty^\ep$ the \emph{finite difference infinity Laplacian} and the equation the \EQ{finite-difference-infinity-laplace} the \emph{finite difference infinity Laplace equation}.

\begin{remark}
Let us briefly mention that the value functions for Players I and II are bounded: this is easy to see by adopting a strategy of ``pulling" toward a specified point on the boundary. This strategy forces the game to terminate after at most $C \ep^{-2}$ expected steps (we refer to \cite{Peres:2009} for details).
\end{remark}

Our approach is to study the finite difference equation \EQ{finite-difference-infinity-laplace} directly, using PDE methods. While we use the underlying tug-of-war games for intuition and motivation, we make no probabilistic arguments in this paper (with the exception of the proof of \LEM{probability-martingale}). Several of our analytic techniques are suggested by probabilistic arguments in \cite{Peres:2009}, see for example the discussion preceding \LEM{max-over-balls}.

\subsection*{Main results} Our first theorem establishes comparison for solutions of the finite difference infinity Laplace equation. In order to state this result, we require the following definition.

\begin{definition}
Let $\ep > 0$ and $v:\bar \Omega \to \R$. We say that $x\in \Omega$ is a \emph{strict $\ep$-local maximum} of $v$ if there is a closed set  $F \subseteq \Omega$ with $x \in F$ and $v(x) = \sup_F v > v(y)$ for every $y\in F^\ep \setminus F$, where we denote $F^\ep := \{ x \in \bar \Omega : \dist(x, F) \leq \ep \}$. Similarly, we say that $x\in \Omega$ is a \emph{strict $\ep$-local minimum} of $u$ if $x$ is a strict $\ep$-local maximum of $-u$.
\end{definition}

\begin{thm}
\label{thm:no-extrema-comparison}
Assume that the functions $u, -v \in \USC(\bar\Omega)$ satisfy
\begin{equation}\label{eq:no-extrema-comparison-ineq}
- \Delta_\infty^\ep u \leq - \Delta_\infty^\ep v \quad \mbox{in} \  \Omega.
\end{equation}
Suppose also that $u$ has no strict $\ep$-local maximum, or $v$ has no strict $\ep$-local minimum, in $\Omega$. Then
\begin{equation}\label{eq:comparison-conclusion}
\max_{\bar \Omega} (u-v) = \max_{\partial \Omega} (u-v).
\end{equation}
\end{thm}

We show in \LEM{no-extrema-no-extrema} that $u \in \USC(\bar\Omega)$ and $- \Delta_\infty^\ep u \leq 0$ in $\Omega$ imply that $u$ has no strict $\ep$-local maximum in $\Omega$. By symmetry, we deduce that $v \in \LSC(\bar\Omega)$ and $- \Delta_\infty^\ep v \geq 0$ in $\Omega$ imply that $v$ has no strict $\ep$-local minimum in $\Omega$. From these observations we immediately obtain the following corollary.

\begin{cor} \label{cor:comparison}
Assume that $u,-v\in \USC(\bar\Omega)$ satisfy the inequality
\begin{equation}
\label{eq:comparison-hyp}
-\Delta^\ep_\infty u \leq -\Delta^\ep_\infty v \quad \mbox{in} \ \Omega.
\end{equation}
Suppose also that $-\Delta_\infty^\ep u \leq 0$ or $-\Delta_\infty^\ep v \geq 0$ in $\Omega$. Then \EQ{comparison-conclusion} holds.
\end{cor}

Our next main result establishes the existence of solutions. In fact, we show that the Dirichlet problem for the finite difference possesses unique maximal and minimal solutions, which are the value functions for Players II and I, respectively, for boundary-biased tug-of-war.

\begin{thm}\label{thm:existence}
For each $\ep > 0$, there exist solutions $\underline u_\ep, \bar u_\ep \in C(\bar\Omega)$ of the equation
\begin{equation}\label{eq:existence}
\left\{   \begin{aligned}
& -\Delta^\ep_\infty u = f & & \mbox{in} \ \Omega, \\
& u = g & & \mbox{on} \ \partial \Omega,
\end{aligned} \right.
\end{equation}
with the property that if $w:\bar\Omega \to \R$ is any bounded function satisfying the inequalities
\begin{equation}
\left\{   \begin{aligned}
&-\Delta_\infty^\ep w \leq (\geq) \ f & & \mbox{in} \ \Omega, \\
& w \leq (\geq) \ g & & \mbox{on} \ \partial \Omega,
\end{aligned} \right.
\end{equation}
then $w \leq \bar u_\ep$ ($w \geq \underline u_\ep$) on $\bar \Omega$. Moreover, $\underline u_\ep$ is the value function for Player I, and $\bar u_\ep$ is the value function for Player II for  the corresponding boundary-biased $\ep$-step tug-of-war game.
\end{thm}

It is not known if standard $\ep$-turn tug-of-war possesses a value if $f \geq 0$, $f\not \equiv 0$, and $\inf f = 0$, or if $f$ fails to be uniformly continuous. In contrast, according to \COR{comparison} and \THM{existence}, if $f\in C(\Omega) \cap L^\infty(\Omega)$ is nonnegative or nonpositive, then the problem \EQ{existence} has a unique solution $\underline u_\ep = \bar u_\ep$, which is the value of the corresponding boundary-biased $\ep$-step tug-of-war game. \THM{no-extrema-comparison} provides uniqueness in even greater generality: if $\underline u_\ep \not \equiv \bar u_\ep$, then $\underline u_\ep$ has a strict $\ep$-local minimum and $\bar u_\ep$ has a strict $\ep$-local maximum.

This latter result has an interesting probability interpretation. In \cite{Peres:2009}, it was shown that nonuniqueness of solutions may arise from the necessity of guaranteeing termination of the game. If Player I must select his strategy first, then he must ensure that the game terminates after finitely many steps with probability 1, and likewise if Player II chooses her strategy first, she must ensure termination. In certain cases, the player selecting first may be required to adopt a strategy which gives up favorable positions in order to ensure termination of the game. One might suspect that unless there is a good reason for each player to keep the token away from the boundary (e.g., the value functions have a strict $\ep$-local maximum or minimum), then the previous situation does not arise, since one of the players would ensure termination simply by playing optimally. In the latter case, we expect the value functions for the two players to be equal. \THM{no-extrema-comparison} is a formal justification of this intuition.

One might also suspect that if the terminal payoff function $g$ has large oscillation relative to $\| f \|_{L^\infty(\Omega)}$, then the players should aim for a favorable spot on the boundary rather than concern themselves with accumulating running payoff. Thus perhaps in this case the value functions for the players have no strict $\ep$-local extrema, and hence the game has a value. As a further application of \THM{no-extrema-comparison}, we obtain the following uniqueness result for sign-changing but small $f$ and nonconstant $g$, which rigorously justifies to this informal heuristic.

\begin{thm}\label{thm:small-f-uniqueness}
Assume that $\Omega$ is convex and for each $x_0 \in \partial \Omega$ and $r > 0$, the function $g$ is not constant on $\partial \Omega \cap B(x_0,r)$. For each $\ep > 0$ there exists a constant $\delta > 0$ such that if $\| f \|_{L^\infty(\Omega)} \leq \delta$, then the boundary-value problem \EQ{existence} has a unique solution.
\end{thm}

We give two proofs of the following result, which asserts that as $\ep \to 0$ solutions of the finite difference infinity Laplace equation converge to a solution of the continuum infinity Laplace equation. It is an analogue of the last statement in \THM{Peres2009} for the value functions of boundary-biased tug-of-war. Our result is more general, as we impose no assumptions on $f \in C(\Omega)$.

\begin{thm}\label{thm:convergence}
Assume only that $f \in C(\Omega)$, and that $\{ \ep_j \}_{j=1}^\infty$ is a sequence of positive numbers converging to $0$ as $j \to \infty$, such that for each $j$ the function $u_j \in C(\Omega)$ is a viscosity subsolution of the inequality
\begin{equation}\label{eq:convergence}
-\Delta_\infty^{\ep_j} u_j \leq f \quad \mbox{in} \ \Omega_{\ep_j}.
\end{equation}
Suppose also that there exists a function $u \in C(\Omega)$ such that $u_j \to u$ locally uniformly in $\Omega$ as $j\to \infty$. Then $u$ is a viscosity subsolution of the inequality
\begin{equation} \label{eq:convergence-viscosity}
-\IL u \leq f \quad \mbox{in} \ \ \Omega.
\end{equation}
\end{thm}

We now turn to results for the continuum equation, which we obtain with the help of the results above and an interesting relationship between solutions of the the continuum and discrete equations (see \PROP{continuous-sup-convolution} below).

From \PROP{equicontinuous} below, we see that if $\ep_j\downarrow 0$, then a sequence $\{ u_j \} \subseteq C(\bar\Omega)$ of solutions of \EQ{existence} is uniformly equicontinuous. Such a sequence $\{ u_j \}$ is also uniformly bounded. In particular, the Arzela-Ascoli theorem asserts that every sequence $\{ \ep_j \}$ has a subsequence for which the maximal solutions of \EQ{existence} for $\ep = \ep_j$ converge uniformly on $\bar \Omega$ to some function $u$. According to \THM{convergence}, the limit function $u$ is a viscosity solution of
\begin{equation}\label{eq:existence-continuous}
\left\{ \begin{aligned}
& -\Delta_\infty u = f & \mbox{in} & \ \Omega, \\
& u = g & \mbox{on} & \ \partial \Omega.
\end{aligned} \right.
\end{equation}
In particular, the boundary-value problem \EQ{existence-continuous} possesses a solution for any given $f\in C(\Omega) \cap L^\infty(\Omega)$ and $g\in C(\partial \Omega)$. This result appears to be new for the normalized infinity Laplacian, as all previous existence results of which we are aware (see for example \cite[Theorems 4.1 and 4.2]{Lu:2008}, in addition to \THM{Peres2009} above) have required $f > 0$, $f < 0$ or $f \equiv 0$ in $\Omega$.

\begin{cor}\label{cor:existence-continuous}
There exists a viscosity solution $u \in C(\bar\Omega)$ of \EQ{existence-continuous}.
\end{cor}

The following stability result is a generalization of \cite[Theorem 1.9]{Lu:2008}. The latter result imposes the additional assumption that $f$ and $f_k$ be positive, negative, or identically zero in $\Omega$.

\begin{thm}\label{thm:stability}
Assume that $f\in C(\Omega)\cap L^\infty(\Omega)$ and $\{f_k\}_{k=1}^\infty \subseteq C(\Omega)\cap L^\infty(\Omega)$ such that $\sup_{k\geq 1} \| f_k \|_{L^\infty(\Omega)} < \infty$ and $f_k \to f$ locally uniformly in $\Omega$ as $k\to \infty$. Suppose that for each $k$, the function $u_k \in C(\bar\Omega)$ is a viscosity solution of the problem
\begin{equation}
\left\{ \begin{aligned}
& -\Delta_\infty u_k = f_k & \mbox{in} & \ \Omega, \\
& u_k = g & \mbox{on} & \ \partial \Omega.
\end{aligned} \right.
\end{equation}
Then there exist a subsequence $\{ u_{k_j} \}$ and a solution $u\in C(\bar\Omega)$ of \EQ{existence-continuous} such that $u_{k_j} \to u$ uniformly on $\bar \Omega$ as $j \to \infty$.
\end{thm}

With the help of \THM{stability} we obtain the following existence result, which is an improvement of \COR{existence-continuous}.

\begin{thm}\label{thm:existence-continuous}
There exist solutions $\underline u, \bar u \in C(\bar \Omega) \cap C^{0,1}(\Omega)$ of \EQ{existence-continuous} such that whenever $w \in USC(\bar \Omega)$ ($w \in LSC(\bar \Omega)$) is a subsolution (supersolution) of the equation $-\Delta_\infty w = f$ and $w\leq g$ ($w \geq g$) on $\partial \Omega$, we have $w \leq \bar u$ ($w \geq \underline u$) in $\Omega$.
\end{thm}

Recently, Lu and Wang \cite{Lu:preprint} found another proof of \THM{existence-continuous} using a different approach.

In the case that $\underline u = \bar u$, problem \EQ{existence-continuous} has a unique solution, and we immediately deduce that $\| \bar u_\ep - \underline u_\ep \|_{L^\infty(\Omega)} \to 0$ as $\ep \to 0$, since both $\bar u_\ep$ and $\underline u_\ep$ must converge to this unique solution, uniformly on $\bar \Omega$.

\begin{cor}\label{cor:uniqueness-convergence}
Assume that the boundary-value problem \EQ{existence-continuous} has a unique viscosity solution $u \in C(\bar\Omega)$. Then $\bar u_\ep \to u$ and $\underline u_\ep \to u$ uniformly on $\bar \Omega$ as $\ep \to 0$.
\end{cor}

Our next result asserts that uniqueness occurs in the generic case, which gives an affirmative answer to the first open question posed in Section 8 of \cite{Peres:2009}. It is easily deduced from \THM{existence-continuous} and \PROP{strict-comparison}, below.

\begin{thm}\label{thm:uniqueness-generic}
There exists an at most countable set $\mathcal{N} \subseteq \R$ such that the problem 
\begin{equation}
\left\{   \begin{aligned}
&-\Delta_\infty u = f + c & & \mbox{in} \ \Omega, \\
& u = g & & \mbox{on} \ \partial \Omega,
\end{aligned} \right.
\end{equation}
has a unique solution for every $c \in \R \setminus \mathcal{N}$.
\end{thm}

Via similar arguments we obtain the corresponding statement for the discrete infinity Laplace equation.

\begin{thm}\label{thm:uniqueness-generic-discrete}
There exists an at most countable set $\mathcal{N}_\ep \subseteq \R$ such that the problem 
\begin{equation}
\left\{   \begin{aligned}
&-\Delta^\ep_\infty u = f + c & & \mbox{in} \ \Omega, \\
& u = g & & \mbox{on} \ \partial \Omega,
\end{aligned} \right.
\end{equation}
has a unique solution for every $c \in \R \setminus \mathcal{N}_\ep$.
\end{thm}

Examples are presented in \cite[Section 5]{Peres:2009} of $f$ for which the boundary value problem \EQ{existence-continuous} has infinitely many solutions. The functions $f$ given in these examples change sign in $\Omega$. This non-uniqueness phenomenon is not well understood. It is even unknown whether we have uniqueness for \EQ{existence-continuous} under the assumption that $f\geq 0$. The most general uniqueness result is the following theorem, which first appeared\footnote{Although this is a simple corollary of \THM{Peres2009}, since we only need to move in a little from the boundary.} in \cite{Lu:2008}.

\begin{thm}[See \cite{Lu:2008} and \cite{Peres:2009}]\label{thm:comparison-continuous}
Assume that $f> 0$, $f< 0$, or $f\equiv 0$. Suppose that $u,-v\in \USC(\bar\Omega)$ satisfy the differential inequalities
\begin{equation}
-\Delta_\infty u \leq f \leq -\Delta_\infty v\quad \mbox{in} \ \Omega.
\end{equation}
Then
\begin{equation*}
\max_{\bar \Omega} (u-v) = \max_{\partial \Omega} (u-v).
\end{equation*}
\end{thm}

The uniqueness of infinity harmonic functions with given boundary data is due to Jensen \cite{Jensen:1993}, and new proofs and extensions have appeared in the papers of Barles and Busca \cite{Barles:2001}, Aronsson, Crandall, and Juutinen \cite{Aronsson:2004}, Crandall, Gunnarsson and Wang \cite{Crandall:2007}, Peres, Schramm, Sheffield, and Wilson \cite{Peres:2009}, and Lu and Wang \cite{Lu:2008}. With the exception of \cite{Peres:2009}, which used probabilistic methods, all of the papers mentioned above use deep results in viscosity solution theory (as presented for example in \cite{UsersGuide}) as well as Aleksandrov's theorem on the twice differentiability of convex functions.

Recently, the authors \cite{Armstrong:preprint} discovered a new proof of the uniqueness of infinity harmonic functions which does not invoke the uniqueness machinery of viscosity solution theory or Aleksandrov's theorem. Here we generalize the argument presented in \cite{Armstrong:preprint} to give a new PDE proof of \THM{comparison-continuous}. Our argument uses only results for the finite difference infinity Laplace equation, and \PROP{continuous-sup-convolution}, below.

The next theorem is an explicit estimate of the difference between an infinity harmonic function and a solution of the infinity Laplace equation with small right-hand side, relative to fixed boundary data. Our argument is a combination of the methods we develop here with the ``patching lemma" of Crandall, Gunnarsson, and Wang \cite[Theorem 2.1]{Crandall:2007}.

\begin{thm} \label{thm:dependence-zero}
For each $\gamma \geq 0$, let $u_\gamma \in C(\bar \Omega)$ denote the unique solution of the problem
\begin{equation*}
\left\{ \begin{aligned}
& - \Delta_\infty u_\gamma = \gamma & & \mbox{in} \ \Omega, \\
& u_\gamma = g & & \mbox{on} \ \partial \Omega.
\end{aligned} \right.
\end{equation*}
There exists a constant $C>0$, depending only on $\diam(\Omega)$, such that
\begin{equation}
\label{eq:dep-zero}
\| u_\gamma - u_0 \|_{L^\infty(\bar \Omega)} \leq C \gamma^{\frac{1}{3}} \left( 1 + \|  g \|_{L^\infty(\partial \Omega)}  \right).
\end{equation}
\end{thm}

As an application of \THM{dependence-zero}, we deduce an upper bound for the expected duration of a game of boundary-biased tug-of-war.

\begin{cor}\label{cor:duration-game}
For any given $\delta, \ep > 0$, in boundary-biased $\ep$-step tug-of-war with no running payoff, Player I has a strategy that achieves an expected payoff of at least $V^\ep_{I}(x_0) - \delta$, for any initial point $x_0\in \Omega$, and for which the expected number of stages it takes for the game to terminate is less than $C\delta^{-3} \ep^{-2}$. The constant $C> 0$ depends only on the oscillation $\max_{\partial \Omega}  g - \min_{\partial \Omega} g$ of the boundary data and the domain $\Omega$.
\end{cor}

The connection between \THM{dependence-zero} and \COR{duration-game} follows from an observation of Peres, Pete, and Somersille \cite[Proposition 7.1]{Peres:preprint}, who proved \COR{duration-game} with upper bound $C(\delta) \ep^{-2}$, using a stability result of Lu and Wang \cite{Lu:2008}. We do not give the proof of \COR{duration-game} here, and instead refer the reader to the discussion in \cite{Peres:preprint}.

Let us mention that we can generalize \COR{duration-game} to any running payoff function $f\in C(\Omega) \cap L^\infty(\Omega)$ for which we have $\underline u_\ep \rightarrow \underline u$. In this case, we deduce an upper bound of the form $C(\delta,f,g) \ep^{-2}$ on the expected number of stages before termination, in boundary-biased $\ep$-step tug-of-war, for some fixed Player I strategy that is expected win at least $V^\ep_{I} - \delta$.

\subsection*{Overview of this article}
\SEC{three} is devoted to the study of the finite difference equation $-\Delta_\infty^\ep u = f$, where we prove Theorems \ref{thm:no-extrema-comparison}, \ref{thm:existence}, and \ref{thm:small-f-uniqueness} and study the regularity of solutions. In \SEC{four} we give our first proof of \THM{convergence}. In \SEC{five}, we apply techniques developed in the previous sections to the continuum equation $-\Delta_\infty u = f$. We give a second proof of \THM{convergence}, a new elementary proof of the uniqueness of infinity harmonic functions (a special case of \THM{comparison-continuous}), and prove Theorems \ref{thm:stability}, \ref{thm:existence-continuous}, \ref{thm:uniqueness-generic}, and \ref{thm:uniqueness-generic-discrete}. \SEC{six} is devoted to the relationship between continuous dependence of solution of the finite difference equation and uniqueness for the continuum equation. There we complete the proof of \THM{comparison-continuous} and prove \THM{dependence-zero}, as well as obtain explicit estimates for the rate of convergence, as $\ep \to 0$, of solutions of the finite difference equation to those of the continuum equation. In \SEC{seven} we highlight some interesting open problems.


\section{The finite difference infinity Laplacian}\label{sec:three}

In this section, we study the solutions $u$ of the difference equation
\begin{equation}
- \Delta_\infty^\ep u = f \mbox{ in } \Omega
\end{equation}
subject to the Dirichlet boundary condition
\begin{equation}
u = g \quad \mbox{on}  \ \partial \Omega.
\end{equation}

\begin{remark}\label{rem:strict-comparison-discrete}
We employ the following simple observation many times in this section. If $u,v:\bar\Omega \to \R$ are bounded functions and $x_0 \in \Omega$ is such that
\begin{equation*}
(u-v)(x_0)  = \sup_{\bar \Omega(x,\ep)} (u-v),
\end{equation*}
then
\begin{equation*}
S^+_\ep u(x_0) \leq S^+_\ep v(x_0) \quad \mbox{and} \quad S^-_\ep u(x_0) \geq S^-_\ep v(x_0).
\end{equation*}
In particular,
\begin{equation*}
-\Delta_\infty^\ep u(x_0) \geq -\Delta_\infty^\ep v(x_0).
\end{equation*}
\end{remark}

Generalizing an argument of Le Gruyer \cite{LeGruyer:2007}, who established the uniqueness of solutions of a difference equation on a finite graph, we prove \THM{no-extrema-comparison}.

\begin{proof}[{\bf Proof of \THM{no-extrema-comparison} }]
Assume that $u,-v\in \USC(\bar\Omega)$ satisfy the inequality
\begin{equation}\label{eq:no-extrema-comparison-eq1}
-\Delta_\infty^\ep u \leq -\Delta_\infty^\ep v \quad \mbox{in} \ \Omega,
\end{equation}
but
\begin{equation}
m: = \max_{\bar\Omega} (u-v) > \max_{\partial \Omega} (u-v).
\end{equation}
By symmetry, we need to show only that $u$ has a strict $\ep$-local maximum in $\Omega$.
Define the set
\begin{equation*}
E:= \left\{ x\in \bar\Omega : (u-v) (x) = m \right\}.
\end{equation*}
The set $E$ is nonempty, closed, and contained in $\Omega$. Let $l := \max_E u$. Since $u$ is upper semicontinuous, the set
\begin{equation*}
F:= \{ x \in E : u(x) = l \}
\end{equation*}
is nonempty and closed. From \REM{strict-comparison-discrete} and the inequality \EQ{no-extrema-comparison-eq1}, we see that
\begin{equation}\label{eq:no-extrema-eq-1}
S^-_\ep u(x) = S^-_\ep v(x) \quad \mbox{and} \quad S^+_\ep u(x) = S^+_\ep v(x) \quad \mbox{for every} \ x \in E.
\end{equation}

We claim that every point of $F$ is a strict $\ep$-local maximum of $u$. We need to show that 
\begin{equation*}
u(y) < l \quad \mbox{for every} \ y \in F^\ep \setminus F.
\end{equation*}
Suppose on the contrary that there is a point $y \in \bar \Omega \setminus F$ such that $\dist(y,F) \leq \ep$ and $u(y) \geq l$. It follows that $y \not \in E$, by the definition of $F$. Pick $x \in F$ with $y \in \bar \Omega(x,\ep)$. By reselecting $y$, if necessary, we may assume that
\begin{equation*}
S^+_\ep u(x) = \frac{u(y) - u(x)}{\rho_\ep(x,y)}.
\end{equation*}
Since $y\not \in E$, we see that
\begin{equation*}
u(y) - v(y) < m = u(x) - v(x). 
\end{equation*}
Thus
\begin{equation*}
S^+_\ep u(x) = \frac{u(y) - u(x)}{\rho_\ep(x,y)} < \frac{v(y) - v(x)}{\rho_\ep(x,y)} \leq S^+_\ep v(x).
\end{equation*}
This contradicts the second equality of \EQ{no-extrema-eq-1}, completing the proof.
\end{proof}

In order to obtain \COR{comparison} from \THM{no-extrema-comparison}, we need the following lemma.

\begin{lem}\label{lem:no-extrema-no-extrema}
Suppose that $u \in \USC(\bar \Omega)$ satisfies the inequality
\begin{equation}\label{eq:no-extrema-no-extrema-eq-1}
-\Delta^\ep_\infty u \leq 0 \quad \mbox{in} \ \Omega.
\end{equation}
Then $u$ has no strict $\ep$-local maximum in $\Omega$. 
\end{lem}
\begin{proof}
Suppose on the contrary that $u$ has a strict $\ep$-local maximum at $x_0 \in \Omega$. Select a nonempty closed set $F \subseteq \Omega$ which contains $x_0$ and for which
\begin{equation}\label{eq:no-extrema-no-extrema-eq-2}
u(x_0) = \max_F u > u(y) \quad \mbox{for every} \ y \in F^\ep \setminus F.
\end{equation}
Set $E := \{ y \in F : u(y) = u(x_0) \}$. Then $E$ is nonempty, closed, and for any $y \in \partial E$ we have
\begin{equation*}
S^+_\ep u(y) = 0 < S^-_\ep u(y),
\end{equation*}
a contradiction to \EQ{no-extrema-no-extrema-eq-1}.
\end{proof}

By an argument similar to the proofs of \THM{no-extrema-comparison} and \LEM{no-extrema-no-extrema}, we obtain the following proposition.

\begin{prop}\label{prop:no-extrema-comparison-ep-thick}
Assume $u, -v \in \USC(\Omega)$ and $\ep > 0$ satisfy the inequality
\begin{equation}
-\Delta^\ep_\infty u \leq - \Delta^\ep_\infty v \quad \mbox{in} \ \Omega_\ep.
\end{equation}
Suppose that $u$ has no strict $\ep$-local maximum in $\Omega_\ep$, or $v$ has no strict $\ep$-local minimum in $\Omega_\ep$. Then
\begin{equation}\label{eq:no-extrema-comparison-ep-thick-eq}
\sup_{\Omega} (u-v) = \sup_{\Omega \setminus \Omega_\ep} (u-v).
\end{equation}
In particular, \EQ{no-extrema-comparison-ep-thick-eq} holds provided that $-\Delta^\ep_\infty u \leq 0$ or $-\Delta^\ep_\infty v \geq 0$ in $\Omega_\ep$.
\end{prop}

Recall that the \emph{upper semicontinuous envelope} $u^*$ of a bounded function $u: \bar\Omega \to \R$ is defined by
\begin{equation*}
u^*(x) := \limsup_{y\to x} u(y),\quad x\in \bar\Omega.
\end{equation*}
The \emph{lower semicontinuous envelope} of $u$ is $u_* : = - ( -u)^*$. The function $u^*$ is upper semicontinuous, $u_*$ is lower semicontinuous, and $u_* \leq u \leq u^*$.


\begin{lem}\label{lem:uppersemi}
Suppose $h\in \USC(\Omega)$ and $u:\bar \Omega \to \R$ are bounded from above and satisfy the inequality
\begin{equation}
\label{eq:uppersemi-subsol}
- \Delta_\infty^\ep u \leq h  \mbox{ in } \Omega.
\end{equation}
Then $u^*$  also satisfies \EQ{uppersemi-subsol}. If in addition $u \leq g$ on $\partial \Omega$, then $u^* \leq g$ on $\partial \Omega$.
\end{lem}

\begin{proof}
Fix $x \in \Omega$, and let $\{ x_k \}_{k=1}^\infty \subseteq \Omega$ be a sequence converging to $x$ as $k \to \infty$ and for which
\begin{equation*}
u(x_k) \rightarrow u^*(x).
\end{equation*}
We claim that
\begin{equation}
\label{eq:uppersemi-splus-lim}
S^+_\ep u^*(x) \geq  \limsup_{k \rightarrow \infty}  S^+_\ep u(x_k) .
\end{equation}
Fix $\delta > 0$, and select for each $k$ a point $y_k \in \bar \Omega(x_k, \ep)$ such that
\begin{equation}
\label{eq:uppersemi-splus}
S^+_\ep u(x_k) \leq \frac{u(y_k) - u(x_k)}{\rho_\ep(x_k, y_k)} + \delta.
\end{equation}
By taking a subsequence, we may assume that $y_k \rightarrow y \in \bar\Omega(x,\ep)$. In the case that
\begin{equation}
\label{eq:uppersemi-rho}
\rho_\ep(x_k, y_k) \rightarrow \rho_\ep(x, y),
\end{equation}
then we may pass to the limit $k\to \infty$ in \EQ{uppersemi-splus} to obtain
\begin{align*}
S^+_\ep u^*(x) \geq \frac{u^*(y) - u^*(x)}{\rho_\ep(x,y)} \geq \limsup_{k \rightarrow \infty} \frac{u(y_k) - u(x_k)}{\rho_\ep(x_k,y_k)} \geq \limsup_{k \rightarrow \infty}  S^+_\ep u(x_k)  - \delta.
\end{align*}
Sending $\delta \to 0$, we obtain \EQ{uppersemi-splus-lim} provided that \EQ{uppersemi-rho} holds. On the other hand, suppose that \EQ{uppersemi-rho} fails. Then $y \in \partial \Omega \cap \{ z: d(y,z) < \ep\}$ and $y_k \in \Omega(x, \ep) \setminus \partial \Omega$ for infinitely many $k$. By taking a subsequence, assume that $y_k \in \Omega(x, \ep) \setminus \partial \Omega$ for all $k$. Thus $\rho_\ep(x, y_k) = \rho_\ep(x_k, y_k) = \ep$ for every $k$, and we have
\begin{align*}
S^+_\ep u^*(x) & \geq \limsup_{k \rightarrow \infty} \frac{u^*(y_k) - u^*(x)}{\rho_\ep(x,y_k)} \geq \limsup_{k \rightarrow \infty} \frac{u(y_k) - u(x_k)}{\rho_\ep(x_k,y_k)}  \geq \limsup_{k \rightarrow \infty}  S^+_\ep u(x_k)  - \delta.
\end{align*}
Passing to the limit $\delta \to 0$, we have \EQ{uppersemi-splus-lim} also in the case that \EQ{uppersemi-rho} fails.

To derive a similar estimate for $S^-_\ep u^*(x)$, let $\delta > 0$ and select a point $z \in \bar \Omega(x, \ep)$ for which
\begin{equation*}
S^-_\ep u^*(x) \leq \frac{u^*(x) - u^*(z)}{\rho_\ep(x,z)} + \delta,
\end{equation*}
and select any $z_k \in \bar \Omega(x_k, \ep)$ such that $z_k \rightarrow z$. If $z\in \partial \Omega \cap \{ w : d(w,x) < \ep\}$, then we may select $z_k \in \partial \Omega$ for sufficiently large $k$. Then $\rho_\ep(x_k,z_k) \rightarrow \rho_\ep(x,z)$, and we obtain
\begin{equation*}
S^-_\ep u^*(x) \leq \frac{u^*(x) - u^*(z)}{\rho_\ep(x,z)} + \delta \leq \liminf_{k\to \infty} \frac{u(x_k) - u(z_k)}{\rho_\ep (x_k,z_k)} + \delta \leq \liminf_{k\to \infty} S^-_\ep u(x_k) + \delta.
\end{equation*}
Passing to the limit $\delta \to 0$, we derive
\begin{equation}
\label{eq:uppersemi-sminus-lim}
S^-_\ep u^*(x) \leq \liminf_{k \rightarrow \infty} S^-_\ep u(x_k).
\end{equation}

Combining \EQ{uppersemi-splus-lim} and \EQ{uppersemi-sminus-lim}, we obtain
\begin{equation*}
S^-_\ep u^*(x) - S^+_\ep u^*(x) \leq \liminf_{k \rightarrow \infty} \left( S^-_\ep u(x_k) - S^+_\ep u(x_k) \right) \leq \ep h(x).
\end{equation*}
That is,
\begin{equation*}
-\Delta_\infty^\ep u^*(x) \leq h(x) \quad \mbox{for each} \ x \in \Omega.
\end{equation*}

Suppose now that $u \leq g$ on $\partial \Omega$. We show that $u^* \leq g$ on $\partial \Omega$. Suppose on the contrary that $x_k \in \bar \Omega$ such that $x_k \rightarrow x \in \partial \Omega$ and
\begin{equation}\label{eq:uppersemi-bdry-eq-1}
\gamma:= \lim_{k\to\infty} u(x_k) - g(x) > 0.
\end{equation}
As $g$ is continuous and $u\leq g$ on $\partial \Omega$, by taking a subsequence we may assume that $x_k \in \Omega$. Moreover, \EQ{uppersemi-bdry-eq-1} and $u \leq g$ on $\partial \Omega$ imply that
\begin{equation*}
S^-_\ep u(x_k) \rightarrow \infty \quad \mbox{as} \quad k \rightarrow \infty.
\end{equation*}
By the continuity of $g$, there exist $0 < \delta \leq \ep$ and a large positive integer $K$ such that
\begin{equation*}
u(y) \leq g(y) \leq u(x_k) \quad \mbox{for every} \ y \in \partial \Omega\cap \{ z: d(x,z) \leq \delta\}, \ k \geq K.
\end{equation*}
Since $x_k \to x$ as $k\to\infty$, it follows that
\begin{equation*}
\limsup_{k \rightarrow \infty} S^+_\ep u(x_k)  \leq \limsup_{k \rightarrow \infty} \frac{1}{\delta} \left( \sup_{\bar\Omega} u - u(x_k)\right) \leq \frac{1}{\delta}\left( \sup_{\bar \Omega} u - g(x) \right).
\end{equation*}
Since $u$ is bounded above, the expression on the right side of the above inequality is finite. We deduce that
\begin{equation*}
-\ep \Delta_\infty^\ep u(x_k) = S^-_\ep u(x_k) - S^+_\ep u(x_k) \rightarrow +\infty.
\end{equation*}
Since $h$ is bounded above, we have a contradiction to \EQ{uppersemi-subsol}. Thus \EQ{uppersemi-bdry-eq-1} is impossible. It follows that $u^* \leq g$.
\end{proof}


\begin{lem}
\label{lem:cont}
Suppose that $u \in USC(\bar \Omega)$ satisfies
\begin{equation*}
- \Delta_\infty^\ep u = f \quad \mbox{in} \ \Omega,
\end{equation*}
and
\begin{equation*}
u(x_k) \rightarrow u(x) \quad \mbox{whenever} \ x_k \rightarrow x \in \partial \Omega.
\end{equation*}
Then $u \in C(\bar \Omega)$.
\end{lem}

\begin{proof}
Suppose on the contrary that $u\not\equiv u_*$. Then there exists $x_0 \in \Omega$ such that
\begin{equation}
\label{eq:cont-max}
\gamma: = (u-u_*)(x_0) = \max_{\bar \Omega} (u - u_*) > 0.
\end{equation}
Since the function $u - u_*$ is upper semicontinuous, the set $E := \{ u - u_* \geq \gamma \} = \{ u - u_* = \gamma \}$ is closed. Since $u_*$ is lower semicontinuous, by relabeling $x_0$, if necessary, we may assume that
\begin{equation}\label{eq:cont-min}
u_*(x_0) = \min_E u_*.
\end{equation}
According to \LEM{uppersemi},
\begin{equation*}
- \Delta_\infty^\ep u_* \geq f \quad \mbox{in} \ \Omega.
\end{equation*}
In particular, $-\Delta_\infty^\ep u \leq -\Delta_\infty^\ep u_*$. By \EQ{cont-max} and \REM{strict-comparison-discrete}, we see that
\begin{equation*}
S^-_\ep u(x_0) = S^-_\ep u_*(x_0).
\end{equation*}
According to \EQ{cont-max}, this quantity is greater than $\gamma / \ep$. Since $u_*$ is lower semicontinuous, we may select $y \in \bar\Omega(x_0, \ep)$ such that
\begin{equation*}
\frac{u_*(x_0) - u_*(y)}{\rho_\ep(x_0, y)} = S^-_\ep u_*(x_0) = S^-_\ep u(x_0) \geq \gamma / \ep.
\end{equation*}
Using \EQ{cont-min} and $u_*(y) < u_*(x_0)$, we see that $y \not\in E$. Thus $u(y) - u_*(y) < \gamma$. But this implies that
\begin{equation*}
\frac{u(x_0) - u(y)}{\rho_\ep(x_0, y)} > \frac{u_*(x_0) - u_*(y)}{\rho_\ep(x_0, y)} = S^-_\ep u(x_0),
\end{equation*}
which contradicts the definition of $S^-_\ep u(x_0)$.
\end{proof}

We now construct explicit supersolutions to the finite difference infinity Laplace equation, which we find useful below. 

\begin{lem}
\label{lem:cone-super}
Denote
\begin{equation}
\label{eq:cone-super-varphi}
\varphi(x) :=
\left\{ \begin{array}{ll}
q(\path(x, x_0)) & \mbox{ if } x \in \Omega \\
q(\path(x, x_0)) - \delta & \mbox{ if } x \in \partial \Omega,
\end{array}\right.
\end{equation}
for $q(r) := a + b r - \frac{c}{2} r^2$, where $a, b, c, \delta \in \R$ and $x_0 \in \partial \Omega$. If $b \geq c(\ep + \diam(\Omega)) \geq 0$ and $\delta \geq 0$, then $\varphi$ is a solution of the inequality
\begin{equation*}
- \Delta_\infty^\ep \varphi \geq \min \left\{ c, \ep^{-1} (8 \delta c)^{\frac{1}{2}} \right\} \quad \mbox{in} \ \Omega.
\end{equation*}
\end{lem}

\begin{proof}
Fix $x \in \Omega$. Notice that $q$ is nondecreasing on the interval $[0, \ep + \diam(\Omega)]$, and $\varphi$ is lower semicontinuous on $\bar \Omega$. By Cauchy's inequality,
\begin{equation}\label{eq:crap-cauchy}
\left( 2\delta c\right)^{\frac{1}{2}} \leq \frac{s}{2} c + \frac{1}{s} \delta \quad \mbox{for every}  \ s > 0.
\end{equation}
Choose $z_1 \in \bar \Omega(x, \ep)$ such that
\begin{equation*}
S^+_\ep \varphi(x) = \frac{\varphi(z_1) - \varphi(x)}{\rho_\ep(x, z_1)}.
\end{equation*}
Let $\delta_1 = 0$ if $z_1 \in \Omega$, and $\delta_1 = \delta$ if $z_1 \in \partial \Omega$. We have
\begin{align*}
S^+_\ep \varphi(x) & \leq \frac{q(\path(x_0,x) + \rho_\ep(x,z_1)) - q(\path(x_0,x)) - \delta_1}{\rho_\ep(x,z_1)} \\
& = b -c \path(x,x_0) - \frac{c}{2} \rho_\ep(x, z_1) - \frac{\delta_1}{\rho_\ep(x,z_1)}.
\end{align*}
Considering the possible values for $\delta_1$ and recalling \EQ{crap-cauchy}, we deduce that
\begin{equation*}
S^+_\ep \varphi(x) \leq b - c\path(x,x_0) - \min \left\{ \frac{c}{2}\ep , (2\delta c)^{\frac{1}{2}} \right\}.
\end{equation*}

To get the corresponding inequality for $S^-_\ep \varphi(x)$, we choose $z_2\in \bar \Omega(x,\ep)$ along some minimal-length path between $x_0$ and $x$, so that
\begin{equation*}
\path(x,x_0) = \path(x,z_2) + \path(z_2,x_0) \quad \mbox{and} \quad \path(x,z_2) = \min\{ \ep , \path(x,x_0) \} = \rho_\ep(x,z_2).
\end{equation*}
Set $\delta_2 = 0$ if $z_2 \in \Omega$, and $\delta_2 = \delta$ if $z_2 \in \partial \Omega$. We have
\begin{align*}
S^-_\ep \varphi(x) & \geq \frac{\varphi(x) - \varphi(z_2)}{\rho_\ep(x, z_2)} \\
& = \frac{q(\path(x_0,x)) - q(\path(x_0,x) - \rho_\ep(x,z_2)) + \delta_2}{\rho_\ep(x,z_2)} \\
& = b - c\path(x,x_0) + \frac{c}{2} \rho_\ep(x, z_2) + \frac{\delta_2}{\rho_\ep(x,z_2)} \\
& \geq b - c\path(x,x_0)+ \min\left\{ \frac{c}{2} \ep, (2\delta c)^{\frac{1}{2}} \right\}.
\end{align*}
Combining the estimates for $S^+_\ep \varphi(x)$ and $S^-_\ep \varphi(x)$, we obtain
\begin{equation*}
S^-_\ep \varphi(x) - S^+_\ep \varphi(x) \geq \ep \min \{ c , \ep^{-1} (8 \delta c)^{\frac{1}{2}} \}.
\end{equation*}
We divide this inequality by $\ep$ to obtain the lemma.
\end{proof}


In the following lemma, we compare subsolutions of the finite difference equation to the value function for Player II for boundary-biased tug-of-war. This is the only place we employ probabilistic methods in this article.

\begin{lem}\label{lem:probability-martingale}
Assume that $u: \bar \Omega \to \R$ is a bounded function satisfying the inequality
\begin{equation*}
- \Delta_\infty^\ep u \leq f \mbox{ in } \Omega.
\end{equation*}
Then $u \leq g$ on $\partial \Omega$ implies that $u \leq V_{II}$ in $\Omega$, where $V_{II}$ is the value function for Player II with respect to boundary-biased tug-of-war, with running payoff function $f$ and terminal payoff function $g$.
\end{lem}

\begin{proof}
By replacing $u$ with $u^*$ and applying \LEM{uppersemi}, we may assume that $u \in \USC(\bar\Omega)$. We must show that
\begin{equation*}
\inf_{\sigma_{II} \in S} \sup_{\sigma_I \in S} F_{II}(\sigma_I, \sigma_{II}) \geq u,
\end{equation*}
where $S$ is the set of all admissible strategies. It is enough to show that for any fixed strategy $\sigma_{II} \in S$, constant $c > 0$, and starting point $x_0 \in \Omega$, there is a strategy $\sigma_I \in S$ such that
\begin{equation*}
F_{II}(\sigma_I, \sigma_{II})(x_0) \geq u(x_0) - c.
\end{equation*}
Fix a point $y \in \partial \Omega$ and let $\sigma^*_I$ be a strategy such that
\begin{equation*}
\dist_\Omega(\sigma^*_I(x_0, ..., x_k), y) = \min_{\bar \Omega(x_k, \ep)} \dist_\Omega(\cdot, y).
\end{equation*}
That is, $\sigma^*_I$ pulls towards the point $y$. Since $\Omega$ has finite path diameter, we have
\begin{equation*}
m^* := \inf_\Omega F_{II}(\sigma_I^*, \sigma_{II}) > - \infty.
\end{equation*}
We now define a family of approximately optimal strategies for player I. For each $N \in \mathbb{N} \cup \{ \infty \}$, let $\sigma_I^N \in S$ be a strategy for player I such that
\begin{equation*}
\frac{u(\sigma_I^N(x_0, ..., x_k)) - u(x_k)}{\rho_\ep(x_k,\sigma_I^N(x_0, ..., x_k))} = S^+_\ep u(x_k) \quad \mbox{if } k < N,
\end{equation*}
and $\sigma_I^N(x_0, ...., x_k) = \sigma_I^*(x_0,....,x_k)$ if $k \geq N$. That is, $\sigma_I^N$ plays according to $u$ until the $N$th stage and then switches to $\sigma^*_I$.

Consider the stochastic process given by the pair of strategies $(\sigma_I^N, \sigma_{II})$ and the starting point $x_0$. We defined the strategy $\sigma_I^N$ so that the sequence of random variables
\begin{equation*}
M_k^N := u(x_k) + \frac{1}{2} \sum_{j = 1}^k \ep \rho_\ep(x_{j-1},x_j) f(x_{j-1}) \quad \mbox{for } k \leq N,
\end{equation*}
is a submartingale. Indeed, observe that if $k < N$ and $z := \sigma_I^N(x)$ then
\begin{align*}
\lefteqn{\E\left[ M_{k+1} \mid x_k= x \mbox{ and } M_k = m \right] - m} & \\ 
& \geq \min_{y \in \bar \Omega(x,\ep)} \left\{ \frac{\rho_\ep(x,y)}{\rho_\ep(x,z) + \rho_\ep(x,y)} \left( u(z) - u(x) + \frac{\ep}{2} \rho_\ep(x,z) f(x) \right) \right. \\
& \qquad \qquad  + \left. \frac{\rho_\ep(x,z)}{\rho_\ep(x,z) + \rho_\ep(x,y)} \left( u(y) - u(x) + \frac{\ep}{2} \rho_\ep(x,y)f(x) \right) \right\} \\
& = \min_{y \in \bar \Omega(x,\ep)} \left\{ \frac{\rho_\ep(x,y) \rho_\ep(x,z)}{\rho_\ep(x,z) + \rho_\ep(x,y)} \left( S^+_\ep u(x) - \frac{u(x) - u(y)}{\rho_\ep(x,y)} + \ep f(x) \right) \right\} \\
& \geq \min_{y \in \bar \Omega(x,\ep)} \left\{ \frac{\rho_\ep(x,y) \rho_\ep(x,z)}{\rho_\ep(x,z) + \rho_\ep(x,y)} \left( S^+_\ep u(x) - S^-_\ep u(x) + \ep f(x) \right) \right\} \\
& \geq \min_{y \in \bar \Omega(x,\ep)} \left\{ \frac{\rho_\ep(x,y) \rho_\ep(x,z)}{\rho_\ep(x,z) + \rho_\ep(x,y)} \left( \ep \Delta_\infty^\ep u(x) + \ep f(x) \right) \right\} \\
& \geq 0.
\end{align*}
Let $\tau^N$ be the stopping time for the process. By the optional sampling theorem, we have
\begin{equation*}
\E \left[ M_{\tau^N \land N}^N \right] \geq u(x_0),
\end{equation*}
for all $N < \infty$. We may assume that $\tau^\infty$ is finite almost surely, as otherwise
\begin{equation*}
F_{II}(\sigma_I^\infty, \sigma_{II})(x_0) = \infty > u(x_0).
\end{equation*}
Since $\tau^N$ and $\tau^\infty$ are identical for stages $k < N$, it follows that
\begin{equation*}
\prob\left[\tau^N > N \right] \to 0 \quad \mbox{as } N \to \infty.
\end{equation*}
Thus we may estimate
\begin{align*}
F_{II}(\sigma_I^N, \sigma_{II})(x_0) & \geq \prob\left[\tau^N \leq N\right] \E\left[M^N_{\tau^N \land N}\right]  + \prob\left[ \tau^N > N\right] m^* \\
& \geq \prob\left[\tau^N \leq N\right] u(x_0)  + \prob\left[ \tau^N > N\right] m^*,
\end{align*}
and thus obtain
\begin{equation*}
F_{II}(\sigma_I^N, \sigma_{II})(x_0) \geq u(x_0) - c,
\end{equation*}
for sufficiently large $N > 0$.
\end{proof}


We now prove \THM{existence}, using a simple adaptation of Perron's method.

\begin{proof}[{\bf Proof of \THM{existence}}]
Our candidate for a maximal solution is
\begin{equation*}
u(x) : = \sup \left\{ w(x) : w \in \USC(\bar \Omega) \ \mbox{satisfies} \ -\Delta_\infty^\ep w \leq f \ \mbox{in} \ \Omega, \ w \leq g \ \mbox{on} \ \partial \Omega \right\}.
\end{equation*}
According to \LEM{cone-super}, the admissible set is nonempty and $u$ is bounded below. Also, by varying the parameters $\delta > 0$ and $c>0$ in \LEM{cone-super}, we see that $u = g$ on $\partial \Omega$, and that $u$ is continuous at each boundary point $x \in \partial \Omega$. We may also use \LEM{cone-super} together with \REM{strict-comparison-discrete} to see that $u$ is bounded above. 

Let us verify that $u$ satisfies the inequality
\begin{equation}\label{eq:sup-of-sub-is-sub}
-\Delta_\infty^\ep u \leq f \quad \mbox{in} \ \Omega.
\end{equation}
Fix $x\in \Omega$ and $\delta > 0$, and select a function $w$ such that $-\Delta_\infty^\ep w \leq f$ in $\Omega$ and $w \leq g$ on $\partial \Omega$, and for which
\begin{equation*}
u(x) \leq w(x) + \delta \min\left\{ \ep , \dist(x,\partial \Omega) \right\}.
\end{equation*}
Then we have
\begin{align*}
S^-_\ep u(x) = \sup_{y \in \bar\Omega(x,\ep)} \frac{u(x) - u(y)}{\rho_\ep(x,y)} \leq \sup_{y \in \bar\Omega(x,\ep)} \frac{w(x) - w(y)}{\rho_\ep(x,y)} + \delta = S^-_\ep w(x) + \delta.
\end{align*}
In a similar way, we check that $S^+_\ep u(x) \geq S^+_\ep w(x) - \delta$. Hence
\begin{equation*}
-\Delta^\ep_\infty u(x) \leq -\Delta_\infty^\ep w(x) + 2\delta/\ep \leq f(x) + 2\delta / \ep.
\end{equation*}
We now send $\delta \to 0$ to obtain \EQ{sup-of-sub-is-sub}.

According to \LEM{uppersemi} and the definition of $u$, we have $u \geq u^*$, and thus $u = u^*\in \USC(\bar\Omega)$. We now check that $u$ satisfies the inequality
\begin{equation}\label{eq:sup-of-subs-is-sup}
-\Delta_\infty^\ep u \geq f \quad \mbox{in} \ \Omega.
\end{equation}
Suppose \EQ{sup-of-subs-is-sup} fails to hold at some point $x \in \Omega$. Then
\begin{equation*}
S^-_\ep u(x) - S^+_\ep u(x) = \ep f(x) - \gamma,
\end{equation*}
for some $\gamma >0$. Define
\begin{equation*}
w (y) :=  \begin{cases}
u(x) + \frac{\gamma}{2} \min\{ \ep, \dist(x, \partial \Omega) \} & \mbox{if} \ y= x,  \\
u(y) & \mbox{otherwise.}
\end{cases}
\end{equation*}
It is easy to check that
\begin{equation*}
S^-_\ep w (x) - S^+_\ep w (x) \leq \ep f(x).
\end{equation*}
Moreover, for each $y \in \Omega \setminus \{ x \}$ we have
\begin{equation*}
S^+_\ep w(y) \geq S^+_\ep u(y),
\end{equation*}
and, by the upper semicontinuity of $u$,
\begin{equation*}
S^-_\ep w(y) = S^-_\ep u(y).
\end{equation*}
Thus $w$ satisfies the inequality $-\Delta_\infty^\ep  w \leq f$ in $\Omega$, and $w \leq g$ on $\partial \Omega$. Since $w(x) > u(x)$, we obtain a contradiction to the definition of $u$. Thus $u$ satisfies \EQ{sup-of-subs-is-sup}. According to \LEM{cont}, $u\in C(\bar\Omega)$. 

We have shown that $u$ is a solution of $-\Delta_\infty^\ep u = f$ in $\Omega$.  By construction, $u$ is maximal. By \LEM{probability-martingale}, $u \leq V_{II}$. Since $V_{II}$ is also a bounded, measurable solution, we have $V_{II}^* \leq u$ by \LEM{uppersemi} and the definition of $u$. Hence $u = V_{II}$, and the proof is complete.
\end{proof}


We next prove estimates for solutions of the finite difference equation.

\begin{lem}
\label{lem:bcont}
Suppose that $u: \bar \Omega \to \R$ is bounded and satisfies
\begin{equation*}
\left\{ \begin{aligned}
& - \Delta_\infty^\ep u \leq f && \mbox{in} \ \Omega, \\
& u \leq g && \mbox{on} \ \partial \Omega.
\end{aligned} \right.
\end{equation*}
Then there is a constant $C$ depending only on $\diam(\Omega)$ such that
\begin{equation}
\label{eq:bcont-rhobound}
u(x) - g(x_0) \leq 2 \omega_g(\path(x_0,x)) + C \| f \|_{L^\infty(\Omega)} \rho_\ep(x_0,x),
\end{equation}
for every $x \in \Omega$ and $x_0 \in \partial \Omega$.
\end{lem}

\begin{proof}
Using \LEM{uppersemi}, we may assume that $u$ is upper semicontinuous. Fix $x_0 \in \partial \Omega$ and $x_1 \in \Omega$. Let $\varphi$ be the function in the statement of \LEM{cone-super}, for the constants
\begin{align*}
a & := g(x_0) + \omega_g(\path(x_0,x_1)) + \delta, \\
b & := \frac{w_g(\path(x_0,x_1))}{\path(x_0,x_1)} + (\ep + \diam(\Omega))\| f \|_{L^\infty(\Omega)}, \\
c & := \| f \|_{L^\infty(\Omega)}, \\
\delta & := \ep^2 8^{-1} c.
\end{align*}
Using that $\omega_g$ is a modulus of continuity for $g$ and recalling \EQ{omega-g-concave}, it is straightforward to check that $\varphi \geq g$ on $\partial \Omega$. By \COR{comparison} and \LEM{cone-super}, we have $\varphi \geq u$ in $\Omega$. In particular,
\begin{multline*}
u(x_1) \leq \varphi(x_1) = g(x_0) + 2 \omega_g(\path(x_0,x_1)) \\ + \| f \|_{L^\infty(\Omega)} \left( \textstyle\frac{1}{8}\ep^2 + \ep \path(x_0,x_1) + \diam(\Omega) \path(x_0,x_1) \right),
\end{multline*}
which implies \EQ{bcont-rhobound} for $x = x_1$.
\end{proof}


The next lemma uses a marching argument to obtain an interior continuity estimate. By ``marching," we mean an iterated selection of points $x_{k+1}$ which achieve $S^+_\ep u(x_k)$. This is analogous to following the the gradient flowlines of a subsolution of the continuum equation. See \cite[Section 6]{Crandall:2008} for details.

\begin{lem}
\label{lem:icont}
Suppose that $u: \bar \Omega \to \R$ is bounded and satisfies
\begin{equation*}
\left\{ \begin{array}{ll}
- \Delta_\infty^\ep u = f & \mbox{ in } \Omega \\
u = g & \mbox{ on } \partial \Omega.
\end{array} \right.
\end{equation*}
There exists a constant $C$ depending only on $\diam(\Omega)$ such that
\begin{equation}
\label{eq:icont-bound}
|u(x) - u(y)| \leq  C \left( \frac{\omega_g(r)}{r} + \| f \|_{L^\infty(\Omega)} \right) \rho_\ep(x,y),
\end{equation}
for every $x, y \in \Omega_r$ and $r \geq \ep$.
\end{lem}

\begin{proof}
Suppose that \EQ{icont-bound} fails for $C=c$. Then we may assume there is an $r \geq \ep$ and $x_0 \in \Omega_r$ with
\begin{equation}\label{eq:jkdfhjlasd}
S^+_\ep u(x_0) > c \left( \frac{\omega_g(r)}{r} + \| f \|_{L^\infty(\Omega)} \right).
\end{equation}
Assume first that $u$ is continuous. Having chosen $x_0,x_1,\ldots, x_k \in \Omega$, select $x_{k+1} \in \bar \Omega(x_k, \ep)$ such that
\begin{equation}\label{eq:realizesplus}
\frac{u(x_{k+1}) - u(x_k)}{\rho_\ep(x_k, x_{k+1})} = S^+_\ep u(x_k).
\end{equation}
We halt this process at $k=N$, where $x_N \in \partial \Omega$ or $\ep N \geq \diam(\Omega)$.

Notice that whenever $x_k \in \Omega$,
\begin{equation*}
S^+_\ep u(x_{k}) + \ep f(x_{k}) = S^-_\ep u(x_{k}) \geq \frac{u(x_{k}) - u(x_{k-1})}{\ep} = S^+_\ep u(x_{k-1}).
\end{equation*}
Hence for each $1 \leq k \leq N$ such that $x_k \in \Omega$,
\begin{equation*}
S^+_\ep u(x_k) \geq S^+_\ep u(x_0) - \ep k \| f \|_{L^\infty(\Omega)}.
\end{equation*}

We claim that if $c> 0$ is large enough relative to $\diam(\Omega)$, then $x_N \in \partial \Omega$. Suppose that $\ep N \geq \diam(\Omega)$ and $x_N \in \Omega$. Then
\begin{equation}
\label{eq:icont-eq1}
\begin{aligned}
u(x_N) - u(x_0) & = \sum_{k=1}^N \left( u(x_k) - u(x_{k-1})\right) \\
& = \ep \sum_{k=1}^N S^+_\ep u(x_{k-1}) \\
& \geq \ep N S^+_\ep u(x_0) - \ep^2 \frac{N(N-1)}{2} \| f \|_{L^\infty(\Omega)} \\
& \geq \left(c \diam(\Omega) - \diam(\Omega)^2 \right) \left( \frac{w_g(r)}{r} + \| f \|_{L^\infty(\Omega)} \right).
\end{aligned}
\end{equation}
Thus if $c$ is large enough relative to $\diam(\Omega)$, then we derive a contradiction to the estimate for $\sup_{\Omega} |u|$ deduced from \EQ{bcont-rhobound}. Thus we may assume $x_N \in \partial \Omega$ and $\ep (N-1) \leq \diam(\Omega)$.

Thus
\begin{equation*}
\ep N \geq \sum_{k=1}^n \path(x_k, x_{k-1}) \geq \path(x_0,x_N) \geq \dist(x_0,\partial \Omega) \geq r,
\end{equation*}
and from \EQ{omega-g-concave} we deduce that
\begin{equation*}
\omega_g(\path(x_0,x_N)) \leq \ep N \frac{\omega_g(r)}{r}.
\end{equation*}
Using \LEM{bcont} we see that 
\begin{equation}\label{eq:butisht}
|u(x_0) - u(x_N) | \leq \ep N \frac{\omega_g(r)}{r} + C \| f \|_{L^\infty(\Omega)} 
\end{equation}
Combining \EQ{butisht} with a calculation similar to \EQ{icont-eq1} yields
\begin{equation}
\label{eq:icont-eq2}
2 \ep N \frac{\omega_g(r)}{r} + C \| f \|_{L^\infty(\Omega)} \geq  \ep (N-1) S^+_\ep u(x_0) - \ep^2 \frac{N^2}{2} \| f \|_{L^\infty(\Omega)}.
\end{equation}
Using $\ep N \leq \diam(\Omega)$ and recalling \EQ{jkdfhjlasd}, we obtain a contradiction if $c$ is large enough relative to $\diam(\Omega)$. This completes the proof of \EQ{icont-bound} in the case that $u$ is continuous.

If $u$ is not continuous, then we may only choose $x_{k+1}$ which approximate \EQ{realizesplus}. However, since at most $\lceil \ep^{-1} \diam(\Omega) \rceil$ approximations are required, we can let the error be arbitrarily small.
\end{proof}


Combining the two previous lemmas, we obtain a global continuity estimate with respect to $\rho_\ep(x,y)$.

\begin{lem}\label{lem:epcont}
Assume that $u : \bar \Omega \to \R$ satisfies the equation
\begin{equation*}
\left\{ \begin{aligned}
& -\Delta_\infty^\ep u = f &  \mbox{in} & & \Omega, \\
& u = g & \mbox{on} & & \partial \Omega.
\end{aligned}\right.
\end{equation*}
There exists a modulus $\omega \in C(0,\infty)$, which depends only on $\diam(\Omega)$, the modulus $\omega_g$, and $\| f \|_{L^\infty(\Omega)}$, such that
\begin{equation}
\label{eq:epcont}
|u(x) - u(y)| \leq \omega(\rho_\ep(x,y)) \quad \mbox{for every}  \ x, y \in \bar \Omega.
\end{equation}
\end{lem}

\begin{proof}
Fix $x, y \in \Omega$ and set $r : = \rho_\ep(x,y)^{1/2}$. If $x, y \in \Omega_r$, then \EQ{icont-bound} implies
\begin{equation*}
|u(x) - u(y)| \leq C \left( \omega_g(\rho_\ep(x,y)^{1/2}) + \| f \|_{L^\infty(\Omega)} \rho_\ep(x,y)^{1/2} \right),
\end{equation*}
for some $C$ depending only on $\diam(\Omega)$.
If $x, y \in \bar \Omega \setminus \Omega_{2 r}$, then \EQ{bcont-rhobound} and the triangle inequality imply
\begin{equation*}
|u(x) - u(y)| \leq 2 C \left( \omega_g(2\rho_\ep(x,y)^{1/2}) + \| f \|_{L^\infty(\Omega)} 2 \rho_\ep(x,y)^{1/2} \right) + \omega_g(\path(x,y)),
\end{equation*}
for another $C$ depending only on $\diam(\Omega)$.
If $x \in \Omega_{2 r}$ and $y \in \bar \Omega \setminus \Omega_{r}$, then choose $z \in \Omega_{2 r} \setminus \Omega_{r}$ on a minimal path between $x$ and $y$. Since
\begin{equation*}
2 \rho_\ep(x,y) \geq \rho_\ep(x,z) + \rho_\ep(z,y),
\end{equation*}
we can combine the two previous estimates to obtain \EQ{epcont}.
\end{proof}


In the case $f \equiv 0$ and $g$ is Lipschitz, the proofs of the estimates above yield a bit more. In particular, we show that in this case the solution of the finite difference problem is a minimizing Lipschitz extension of $g$ to $\bar \Omega$.

\begin{prop}
\label{prop:minlip}
Assume $g \in C(\partial \Omega)$ is Lipschitz with constant $K$ with respect to the path metric $\path$. Then the unique solution of the problem
\begin{equation}
\label{eq:minlip-eq}
\left\{ \begin{array}{lll}
- \Delta_\infty^\ep u = 0 & \mbox{in} & \Omega, \\
u = g & \mbox{on} & \partial \Omega,
\end{array} \right.
\end{equation}
is also Lipschitz with constant $K$ with respect to the path metric $\path$.
\end{prop}

\begin{proof}
Suppose $u$ is the solution of \EQ{minlip-eq} and define $\tilde u : \bar \Omega \rightarrow \R$ by
\begin{equation*}
\tilde u(x) := \sup_{y \in \bar \Omega} \left( u(y) - K \path(x,y) \right).
\end{equation*}
Observe that $\tilde u$ is Lipschitz with constant $K$ with respect to the path metric $\path$. By \COR{comparison} it is enough to show that $\tilde u$ is a subsolution of \EQ{minlip-eq}.

By \COR{comparison} and \LEM{cone-super}, we know that $u \leq \varphi$ in $\Omega$, where
\begin{equation}
\label{eq:minlip-best}
\varphi(x) := g(x_0) + K \path(x_0, x),
\end{equation}
and $x_0 \in \partial \Omega$ is arbitrary. In particular, $\tilde u \leq g$ on $\partial \Omega$.

Fix $x \in \Omega$. Let $\delta > 0$ and use the continuity of $u$ to choose $y \in \Omega$ such that
\begin{equation*}
\tilde u(x) \leq u(y) - K \path(x,y) + \delta.
\end{equation*}
Next, choose $y^+ \in \bar \Omega(y, \ep)$ such that
\begin{equation*}
S^+_\ep u(y) \leq \frac{u(y^+) - u(y)}{\rho_\ep(y, y^+)},
\end{equation*}
and choose $x^+ \in \bar \Omega(x, \ep)$ as close as possible to $y^+$. Observe that
\begin{equation*}
\path(x^+, y^+) \leq \path(x,y) \quad \mbox{and} \quad \path(x^+, y^+) + \rho_\ep(x, x^+) \leq \path(x,y) + \rho_\ep(y, y^+).
\end{equation*}

If we repeat the marching argument of \LEM{icont}, and use the inequality $u \leq \varphi$ in place of \EQ{bcont-rhobound} when deriving \EQ{icont-eq2}, then we obtain
\begin{equation*}
S^+_\ep u(y) \leq K.
\end{equation*}
Thus, we may compute,
\begin{align*}
S^+_\ep \tilde u(x) & \geq \frac{\tilde u(x^+) - \tilde u(x)}{\rho_\ep(x,x^+)} \\
& \geq \frac{u(y^+) - u(y)}{\rho_\ep(x,x^+)} - K \frac{d(x^+,y^+) - d(x,y)}{\rho_\ep(x,x^+)} -  \frac{\delta}{\rho_\ep(x,x^+)} \\
& \geq \frac{u(y^+) - u(y)}{\rho_\ep(x,x^+)} - S^+_\ep u(y) \frac{\rho_\ep(y, y^+) - \rho_\ep(x,x^+)}{\rho_\ep(x,x^+)} - \frac{\delta}{\rho_\ep(x,x^+)} \\
& \geq S^+_\ep u(y) - \frac{\delta}{\dist(x, \partial \Omega)}.
\end{align*}
By a symmetric calculation, we obtain
\begin{equation*}
S^-_\ep \tilde u(x) \leq S^-_\ep u(y) + \frac{\delta}{\dist(x, \partial \Omega)},
\end{equation*}
and thus
\begin{equation*}
- \Delta_\infty^\ep \tilde u(x) \leq - \Delta_\infty^\ep u(y) + \frac{2 \delta}{\ep \dist(x, \partial \Omega)} \leq \frac{2 \delta}{\ep \dist(x, \partial \Omega)}.
\end{equation*}
Sending $\delta \to 0$ yields $- \Delta_\infty^\ep \tilde u(x) \leq 0$ and the proposition.
\end{proof}

We conclude this section by proving \THM{small-f-uniqueness}. Our proof is a compactness argument, using \THM{no-extrema-comparison} and the following lemma.

\begin{lem}\label{lem:constant-patch}
Assume that $\Omega$ is convex and $u \in C(\bar \Omega)$ is a solution of
\begin{equation}\label{eq:constant-patch-eq}
-\Delta_\infty^\ep u = 0 \quad \mbox{in} \ \Omega.
\end{equation}
Suppose that $S^+_\ep u(x_0) = 0$ for some $x_0 \in \bar \Omega_\ep$. Then $u$ is constant on $\bar \Omega$.
\end{lem}

\begin{proof}
As $\Omega$ is convex, we have $\bar \Omega(x,\ep) = \bar B(x,\ep) \cap \bar \Omega$ and the path metric $d$ is equal to the Euclidean metric. Define
\begin{equation*}
\alpha : = \max_{x \in \bar B(x_0, \ep /2 )} S^+_\ep u(x).
\end{equation*}
We claim that $\alpha = 0$. Suppose on the contrary that $\alpha > 0$. Then
\begin{equation*}
r : = \min \left\{ 0 \leq s \leq \frac{\ep}{2} : S^+_\ep u(y) = \alpha \ \mbox{for some} \ y \in \bar B(x_0,s) \right\} > 0.
\end{equation*}
Select $y_0$ such that $|x_0 - y_0| = r$ and $S^+_\ep u(y_0) = \alpha$. By \EQ{constant-patch-eq}, $S^-_\ep u(y_0) = \alpha$. Select points $y_+,y_-\in \bar \Omega(y_0,\ep)$ such that 
\begin{equation*}
u(y_+) - u(y_0) = \rho_\ep(y_0,y_+) S^+_\ep u(y_0) \quad \mbox{and} \quad u(y_0) - u(y_-) = \rho_\ep (y_0,y_-) S^-_\ep u(y_0)
\end{equation*}

We claim that
\begin{equation}\label{eq:constant-patch-wts}
\rho(y_0, y_\pm) = |y_0 - y_\pm| = \dist(y_\pm , B(x_0,r)).
\end{equation}
If the first equality in \EQ{constant-patch-wts} fails for $y_+$, then $|y_0 - y_+| < \ep$ and $y_+ \in \Omega$. But then we can find a point $z \in B(x_0,r)$ with $|z - y_+| < \ep$, and we deduce that
\begin{equation*}
\ep S^+_\ep u(z) \geq u(y_+) - u(z) = u(y_+) - u(y_0) = \ep \alpha,
\end{equation*}
a contradiction to the definition of $r$. If the second equality in \EQ{constant-patch-wts} fails for $y_+$, then by a similar argument we find a point $z \in B(x_0,r)$ with $S^+_\ep u(z) \geq \alpha$, a contradiction. Using equation \EQ{constant-patch-eq} and very similar arguments, we establish \EQ{constant-patch-wts} for $y_-$.

According to \EQ{constant-patch-wts}, both $y_+$ and $y_-$ lie on the ray from $x_0$ through $y_0$, and $|y_0 - y_\pm| = \min\{ \ep , \dist(y_0,\partial \Omega)\}$. Due to the convexity of $\Omega$, there is only one such point, and thus $y_+ = y_-$. This obviously contradicts the assumption that $\alpha > 0$. 

It follows that $S^+_\ep u(x) \equiv 0$ on $B(x_0,\ep/2)$. Hence the set
\begin{equation*}
\left\{ x \in \bar \Omega_\ep : S^+_\ep u(x) = 0 \right\}
\end{equation*}
is relatively open and closed in $\bar \Omega_\ep$. Since $\Omega$ is convex, the set $\bar\Omega_\ep$ is convex and hence connected. Thus $S^+_\ep u \equiv 0$ on $\bar \Omega_\ep$. Thus $u$ is constant on $\bar \Omega$.
\end{proof}

\begin{proof}[{\bf Proof of \THM{small-f-uniqueness}}]
Suppose on the contrary that for some $\ep > 0$ there is a sequence $\delta_j \downarrow 0$ and a sequence $\{ f_j \} \subseteq C(\Omega) \cap L^\infty(\Omega)$ such that $\| f_j \|_{L^\infty(\Omega)} \leq \delta_j$, and the boundary-value problem
\begin{equation}\label{eq:small-f-eq}
\left\{   \begin{aligned}
& -\Delta^\ep_\infty u = f_j & & \mbox{in} \ \Omega, \\
& u = g & & \mbox{on} \ \partial \Omega,
\end{aligned} \right.
\end{equation}
has more than one solution. In particular, by Theorems \ref{thm:no-extrema-comparison} and \ref{thm:existence}, there exists a solution $u_j$ of \EQ{small-f-eq} with a strict $\ep$-local maximum point $x_j$. We may further assume that $x_j \to x_0$ as $j\to \infty$. 

According to \LEM{bcont}, we have the estimate $\| u_j \|_{L^\infty(\Omega)} \leq C$ independently of $j$. Thus the functions defined
\begin{equation}\label{eq:small-f-ineq}
u(x) := \limsup_{j\to \infty} u_j(x) \quad \mbox{and} \quad v(x) = \liminf_{j\to \infty} u_j(x), \quad x\in \bar \Omega
\end{equation}
are bounded. It is not difficult to check that $u$ and $v$ satisfy the inequalities
\begin{equation*}
-\Delta_\infty^\ep u \leq 0 \leq -\Delta_\infty^\ep v \quad \mbox{in} \ \Omega.
\end{equation*}
Moreover, by the estimate \EQ{bcont-rhobound} it is clear that $u(y_k) \to u(y)$ and $v(y_k) \to v(y)$ whenever $y_k \to y\in \partial \Omega$. According to \LEM{uppersemi}, we have
\begin{equation*}
-\Delta_\infty^\ep u^* \leq 0 \leq -\Delta_\infty^\ep v_* \quad \mbox{in} \ \Omega,
\end{equation*}
and $u^* = g = v_*$ on $\partial \Omega$. Applying \COR{comparison}, we see that $u \leq u^* \leq v_*\leq v$ on $\bar \Omega$. The definitions \EQ{small-f-ineq} imply that $u \geq v$ on $\bar\Omega$, hence $u\equiv u^* \equiv v_*\equiv v$. Thus $u \in C(\bar\Omega)$ is a solution of
\begin{equation}\label{eq:small-f-eq-2}
\left\{   \begin{aligned}
& -\Delta^\ep_\infty u = 0 & & \mbox{in} \ \Omega, \\
& u = g & & \mbox{on} \ \partial \Omega.
\end{aligned} \right.
\end{equation}
Since $S^+_\ep u_j(x_j) = 0$ for all $j$, we deduce that $S^-_\ep u(x_0) = S^+_\ep u(x_0) = 0$. Thus $u$ is constant on the set $\bar \Omega(x_0,\ep)$. Since $u=g$ on $\partial \Omega$ and we assumed that $g$ is not constant on any neighborhood of $\partial \Omega$, we deduce that $x_0 \in \bar \Omega_\ep$. Applying \LEM{constant-patch}, we deduce that $u$ is constant on $\bar \Omega$, which contradicts our hypothesis on $g$. 
\end{proof}


\section{Convergence}\label{sec:four}

In this section, we give our first proof of \THM{convergence}. While the second proof in \SEC{five} is much faster, this first proof has several advantages. For example, it includes an explicit derivation of the continuum infinity Laplace equation from the finite difference equation by Taylor expansion. Also, it is a PDE analogue of the probabilistic proof of convergence appearing in \cite{Peres:2009}, and is thus of independent interest.

Demonstrating the convergence, as $\ep \to 0$, of the value functions for $\ep$-step tug-of-war to a viscosity solution of the infinity Laplace equation is tricky in the case that $f \not \equiv 0$ due to the singularity of the infinity Laplacian. In \cite{Peres:2009}, this difficulty was overcome using an interesting probabilistic argument. The authors introduced a slightly modified tug-of-war game in which one of the players is given more options, and thus a small advantage. The value functions of these \emph{favored} tug-of-war games are shown to be very close to those of standard tug-of-war, and also possess a monotonicity property along diadic sequences $\{ 2^{-k}\ep \}_{k=1}^\infty$. This monotonicity property was used by the authors of \cite{Peres:2009} to prove convergence.

From our point of view, the source of this monotonicity property is a discrete version of the \emph{increasing slope estimate} (the continuum version of which is \LEM{slope-estimates}, below). Observe that, if $u \in USC(\bar\Omega)$ satisfies $-\Delta^\ep_\infty u \leq f$ and $x,y\in \Omega$ satisfy
\begin{equation*}
\ep S^+_\ep u(x) = u(y) - u(x),
\end{equation*}
then
\begin{equation*}
\ep S^+_\ep u(y) \geq \ep S^-_\ep u(y) - \ep^2 f(y) \geq u(y) - u(x) - \ep^2 f(y) = \ep S^+_\ep u(x) - \ep^2 f(y).
\end{equation*}
This simple fact is the basis for the following analytic analogue of the monotonicity properties possessed by the favored tug-of-war games.

\begin{lem} \label{lem:max-over-balls}
Suppose that $u \in USC(\bar\Omega)$ satisfies
\begin{equation}
-\Delta_\infty^\ep u \leq f \quad \mbox{in} \ \Omega_\ep.
\end{equation}
Denote $\tilde u(x) := \max_{\bar B(x,\ep)} u$ for each $x \in \Omega_\ep$, and $\tilde f(x) := \max_{\bar B(x,2\ep)} f$ for every $x \in \Omega_{2\ep}$. Then the function $\tilde{u}$ satisfies the inequality
\begin{equation}
-\Delta_\infty^{2\ep} \tilde u \leq \tilde f \quad \mbox{in} \ \Omega_{3\ep}.
\end{equation}
\end{lem}
\begin{proof}
Fix $x \in \Omega_{3\ep}$. Choose $y\in \bar B(x,\ep)$ and $z\in \bar B(y,\ep) \subseteq \bar B(x,2\ep)$ so that
\begin{equation*}
\ep S^+_\ep u(x) = u(y) - u(x) \ \ \mbox{and} \ \ \ep S^+_\ep u(y) = u(z) - u(y).
\end{equation*}
Notice that
\begin{equation*}
\ep S^+_\ep u(x) = u(y) - u(x) \leq \ep S^-_\ep u(y) \leq \ep S^+_\ep u(y) + \ep^2 f(y),
\end{equation*}
and similarly $\ep S^+_\ep u(y) \leq \ep S^+_\ep u(z) + \ep^2 f(z)$. Thus
\begin{align*}
2\ep S^+_{2\ep} \tilde{u} (x) & = \max_{\bar B(x,3\ep)} u - \max_{\bar B(x,\ep)} u \\
& = \max_{\bar B(x,3\ep)} u - u(z) + u(z) - u(y)\\
& \geq \ep S^+_\ep u(z) + \ep S^+_\ep u(y) \\
& \geq 2\ep S^+_\ep u(y) - \ep^2 f(z) \\
& \geq 2\ep S^+_\ep u(x) - \ep^2 \left(f(z) + 2f(y)\right) \\
& \geq S^+_\ep u(x) + S^-_\ep u(x) - \ep^2 \left(f(z) + 2f(y) + f(x)\right).
\end{align*}
We also estimate
\begin{align*}
2\ep S^-_{2\ep} \tilde u(x) & = u(y) - \min_{z\in \bar{B}(x,2\ep)} \max_{\bar B(z,\ep)} u \\
& \leq u(y) - u(x) + u(x) - \min_{\bar B(x,\ep)} u \\
& = \ep S^+_\ep u(x) + \ep S^-_\ep u(x).
\end{align*}
Combining these calculations, we see that
\begin{equation*}
S^-_{2\ep} \tilde{u} (x) - S^+_{2\ep} \tilde{u} (x) \leq \frac{4\ep^2}{2\ep} \tilde f(x) = (2\ep) \tilde f(x). \qedhere
\end{equation*}
\end{proof}

Before proving \THM{convergence}, we need to verify that the finite difference infinity Laplacian, $\Delta_\infty^\ep$, is close to the continuum infinity Laplacian, $\Delta_\infty$, away from the singularity of the latter.

\begin{lem}\label{lem:discrete-to-continuous}
For any $\varphi \in C^3( \Omega)$, there exists a constant $C>0$,
depending only on $\varphi$, such that
\begin{equation*}
-\Delta_\infty \varphi(x) \leq -\Delta_\infty^\ep\varphi(x) + C\left(
1+ |D\varphi(x)|^{-1} \right) \ep,
\end{equation*}
for every $x \in \Omega_{2 \ep_0}$ with $D\varphi(x) \neq0$ and $0 < \ep \leq \ep_0$.
\end{lem}

\begin{proof}
Fix such an $x$ and $\ep$, and denote $v = D\varphi(x) /
|D\varphi(x)|$. Let $w \in \bar B(0,1)$ such that $\varphi(x+\ep w) =
\max_{\bar B(x,\ep)} \varphi$. We may Taylor expand to get
\begin{equation}\label{eq:lemma-convergence-taylor}
\begin{aligned}
0 & \leq \varphi(x + \ep w) - \varphi(x + \ep v) \\
& \leq \ep (w-v) \cdot D\varphi(x) + \frac{\ep^2}{2} \left(
\left\langle D^2\varphi(x) w, w \right\rangle - \left\langle
D^2\varphi(x) v, v \right\rangle \right) \\
& \qquad \qquad + C\left\| D^3\varphi \right\|_{L^\infty(\bar
\Omega))} \ep^3 |v-w|. \\
& = \ep |D\varphi(x)| (w-v) \cdot v + \frac{\ep^2}{2} \left\langle D^2\varphi(x) (w-v), w+v \right\rangle \\
& \qquad \qquad + C\left\| D^3\varphi \right\|_{L^\infty(\bar
\Omega))} \ep^3 |v-w|.
\end{aligned}
\end{equation}
We may divide by $\ep |D\varphi(x)|$ and rearrange to get
\begin{equation}\label{eq:dotted}
|v|^2 - v\cdot w \leq C \left| D\varphi(x) \right|^{-1} \ep |v-w|.
\end{equation}
Thus
\begin{equation}
|v-w|^2 = |v|^2 + |w|^2 - 2v\cdot w \leq 2|v|^2 - 2v\cdot w \leq C
|D\varphi(x)|^{-1} \ep |v-w|.
\end{equation}
Hence $|v-w| \leq C |D\varphi(x)|^{-1} \ep$. Inserting this estimate into
\EQ{dotted}, we get
\begin{equation*}
0 \leq v\cdot(v-w) \leq C \left| D\varphi(x) \right|^{-1} \ep^2.
\end{equation*}
In turn, inserting this estimate into \EQ{lemma-convergence-taylor} after
dividing the latter by $|D\varphi(x)|$ yields
\begin{equation*}
\varphi( x + \ep w) - \varphi(x + \ep v) \leq C
\left|D\varphi(x)\right|^{-1} \ep^3.
\end{equation*}
In a similar way, we know, for $\tilde{w} \in B(0,1)$ such that
$\varphi(x- \ep \tilde w) = \min_{\bar B(x,\ep)} \varphi$, that
\begin{equation*}
\varphi( x - \ep v) - \varphi(x - \ep \tilde w) \leq C
\left|D\varphi(x)\right|^{-1} \ep^3.
\end{equation*}
Hence
\begin{align*}
-\Delta_\infty \varphi(x) & \leq \ep^{-2} \left( 2\varphi(x) -
\varphi(x+\ep v) - \varphi(x-\ep v) \right) + C\ep \\
& \leq \ep^{-2} \left( 2\varphi(x) - \varphi(x+\ep w) - \varphi(x-\ep
\tilde{w}) \right) + C\left( 1+ |D\varphi(x)|^{-1} \right) \ep \\
& = -\Delta_\infty^\ep \varphi(x) + C\left( 1+ |D\varphi(x)|^{-1}
\right) \ep. \qedhere
\end{align*}
\end{proof}

\begin{proof}[{\bf First proof of \THM{convergence}}]
Suppose that $\varphi$ is a quadratic polynomial and the map $x \mapsto u(x) - \varphi(x)$ has a strict local maximum at $x = x_0 \in \Omega$. Without loss of generality, we may assume that $x_0 =0$ and $\varphi(x_0) = \varphi(0) = 0$. We may write $\varphi(x) = p \cdot x + \frac{1}{2} \left\langle M x, x \right\rangle$, for $p = D\varphi(0)$ and $M =D^2\varphi(0)$. We must show that
\begin{equation}\label{eq:convergence-viscosity-wts}
-\ILP \varphi(0) \leq f(0).
\end{equation}

We will consider the cases $p\neq 0$ and $p = 0$ separately. In the former case, $|D\varphi|$ is bounded away from zero near the origin. Since $u_j \to u$ locally uniformly, for large enough $j$ we may select $x_j \in \bar B(0,\ep_j)$ such that
\begin{equation*}
u_j(x_j) - \varphi(x_j) = \max_{\bar B(x_j,\ep_j)} (u_j - \varphi),
\end{equation*}
and $x_j \to 0$ as $j\to \infty$. We have
\begin{equation*}
-\Delta_\infty^{\ep_j} \varphi(x_j) \leq -\Delta_\infty^{\ep_j}u_j(x_j) = f(x_j).
\end{equation*}
Using \LEM{discrete-to-continuous} and the upper semicontinuity of $\ILP\varphi$ near $0$, we deduce that
\begin{align*}
-\ILP\varphi(0) = \lim_{j\to\infty} -\ILP\varphi(x_j) \leq \limsup_{j \to \infty} -\Delta_\infty^{\ep_j} \varphi(x_j) \leq \limsup_{j\to\infty} f(x_j) = f(0).
\end{align*}
Thus \EQ{convergence-viscosity-wts} holds, in the case that $p \neq 0$.

We now consider the more subtle case that $p=D\varphi(0) = 0$. We must show that
\begin{equation}\label{eq:convergence-viscosity-wts-2}
- \max_{|w|=1} \left\langle M w , w \right\rangle \leq f(0).
\end{equation}
Fix a large positive integer $k\geq 1$. For each $j$, define the functions
\begin{equation*}
v_{j} (x) : = \max_{\bar B(x, R \ep_j)} u_j, \quad f_{j} (x): = \max_{\bar B(x, 2R \ep_j)} f,
\end{equation*}
where $R := 2^k-1$. According to \LEM{max-over-balls}, for every $j$ we have
\begin{equation}\label{eq:convergence-max-balls}
- \Delta_\infty^{2^k\ep_j} v_{j} \leq f_{j} \quad \mbox{in} \ \Omega_{3 R\ep_j}
\end{equation}
It is clear that $v_{j} \to u$ and $f_{j} \to f$ locally uniformly in $\Omega$ as $j\to \infty$. Thus we may select a sequence $x_j \to x_0$ such that for large enough $j$, the function $v_j - \varphi$ attains its maximum in the ball $\bar B (x_j,2^k\ep_j)$ at $x_j$.

Let $z_j \in \bar B(0,1)$ such that
\begin{equation*}
\max_{\bar B(x_j,2^k\ep_j)} \varphi = \varphi(x_j+2^k\ep_jz_j).
\end{equation*}
We claim that we may choose $z_j$ such that
\begin{equation}\label{eq:convergence-push-to-edge}
 |z_j| \geq 1- 2^{-k}  \quad \mbox{for sufficiently large} \ j\geq 1.
\end{equation}
If $\varphi$ has no strict local maximum in $B(x_j,2^k\ep_j)$, then we may choose $|z_j|=1$. If $\varphi$ does have a strict local maximum in $B(x_j,2^k\ep_j)$, then observe that $M$ is negative definite and the strict local maximum must be at the origin. Thus,
\begin{equation}\label{eq:convergence-local-max}
x_j +2^k \ep_j z_j = 0,
\end{equation}
and $u$ also has a strict local maximum at the origin. For large enough $j$, this implies that $v_j( x) = u(0)$ for every $x\in \bar B(0, R\ep_j)$. These facts imply that $v_j - \varphi$ cannot have a local maximum in the ball $B(0,R \ep_j)$. Thus
\begin{equation*}
|x_j | \geq R \ep_j.
\end{equation*}
Recalling \EQ{convergence-local-max}, we see that
\begin{equation*}
|z_j| =  \frac{|x_j |}{2^k \ep_j} \geq \frac{R}{2^k} = 1 - 2^{-k}.
\end{equation*}

Having demonstrated \EQ{convergence-push-to-edge}, we calculate
\begin{align*}
f_j(x_j) & \geq -\Delta_\infty^{2^k\ep_j} v_j(x_j)  \\
& \geq -\Delta_\infty^{2^k\ep_j} \varphi(x_j) \\
& = \frac{1}{4^k\ep_j^2} \left( 2\varphi(x_j) - \max_{\bar B(x_j, 2^k\ep_j)} \varphi - \min_{\bar B(x_j, 2^k\ep_j)} \varphi \right) \\
& \geq \frac{1}{4^k\ep_j^2} \left( 2\varphi(x_j) - \varphi(x_j + 2^k\ep_j z_j) - \varphi(x_j - 2^k\ep_j z_j) \right) \\
& = \frac{1}{4^k\ep_j^2} \left( \left\langle M x_j, x_j \right\rangle - \frac{1}{2}\left\langle M(x_j + 2^k\ep_j z_j) , x_j + 2^k\ep_j z_j \right\rangle \right. \\
& \qquad \qquad \left. - \frac{1}{2}\left\langle M(x_j - 2^k\ep_j z_j) , x_j - 2^k\ep_j z_j \right\rangle \right).
\end{align*}
Rearranging, we obtain
\begin{equation*}
-\left\langle Mz_j , z_j \right\rangle \leq f_j(x_j).
\end{equation*}
By passing to limits along a subsequence and using \EQ{convergence-push-to-edge}, we deduce that
\begin{equation*}
-\left\langle Mz , z \right\rangle \leq f(0),
\end{equation*}
for some vector $z$ satisfying $1-2^{-k} \leq |z| \leq 1$. In particular, we have
\begin{equation*}
- \max_{|w| = 1} \left\langle Mw, w\right\rangle \leq \left( 1 - 2^{-k} \right)^{-2} f(0).
\end{equation*}
Sending $k \to \infty$, we obtain \EQ{convergence-viscosity-wts-2}, and the proof is complete.
\end{proof}

\begin{remark}
In the recent paper by Charro, Garc\'{\i}a Azorero, and Rossi \cite{Charro:2009}, a proof that the value functions for standard $\ep$-turn tug-of-war converge as $\ep \to 0$ to the unique viscosity solution of \EQ{existence-continuous} is attempted in the case that $f\equiv 0$, and with mixed Dirichlet-Neumann boundary conditions. The authors mistakenly state that the value functions for standard $\ep$-step tug-of-war are continuous, an assumption which is needed in their proof. However, we believe the argument in \cite{Charro:2009} can be repaired with straightforward modifications.
\end{remark}

The following proposition is needed in order to obtain \COR{existence-continuous}.

\begin{prop}\label{prop:equicontinuous}
Let $\{ \ep_j \}_{j=1}^\infty$ be a sequence of positive real numbers converging to $0$, and suppose that for each $j$, the function $u_j \in C(\bar\Omega)$ is a solution of \EQ{convergence} with $u_j = g$ on $\partial \Omega$. Then the sequence $\{ u_j \}$ is uniformly equicontinuous.
\end{prop}
\begin{proof}
Let $\delta > 0$. According to \LEM{epcont}, there exists a modulus $\omega:[0,\infty) \to [0,\infty)$, independent of $j$, such that
\begin{equation*}
|u_j(x) - v_j(y) | \leq \omega\left(  \rho_{\ep_j}(x,y) \right),
\end{equation*}
for all $x, y \in \bar \Omega$. Choose a positive integer $m\geq 1$ so large that $\omega(\ep_j) < \delta$ for all $j \geq m$. Also choose a constant $\gamma > 0$ so small that $|u_j(x) - u_j(y) | < \delta$ whenever $|x-y| < \gamma$ and $1 \leq j \leq m-1$. Set $\eta := \min\{ \gamma, \ep_m\}$. Then clearly $|u_j(x) - u_j(y) | < \delta$ for all $j \geq 1$ and all $|x-y| < \eta$.
\end{proof}

\COR{existence-continuous} is now immediately obtained from Theorems \ref{thm:existence} and \ref{thm:convergence}, \PROP{equicontinuous} and the Arzela-Ascoli theorem.


\section{Applications to the Infinity Laplace Equation} \label{sec:five}

In this section we study solutions of the (continuum) infinity Laplace equation. Before beginning, we extend our definition of viscosity subsolution. In \DEF{viscosity} we defined $- \Delta_\infty u \leq f$ only for continuous $f$. However, the definition makes no use of the continuity of $f$, and thus $- \Delta_\infty u \leq h$ makes sense for any $h : \Omega \rightarrow \R$. We use this for $h \in USC(\Omega)$ below.

The following lemma establishes that solutions of the infinity Laplace equation ``enjoy comparisons with quadratic cones," and is a special case of \cite[Theorem 2.2]{Lu:2008}. A slightly less general observation appeared in \cite{Peres:2009}. The notion of \emph{comparisons with cones} was introduced in the work of Crandall, Evans, and Gariepy \cite{Crandall:2001} for infinity harmonic functions. Our proof is new, and much simpler than the one given in \cite{Lu:2008}.


\begin{lem}
\label{lem:cca}
Suppose that $V$ is a bounded, open subset of $\R^n$, and $u \in USC(\bar V)$ satisfies
\begin{equation}\label{eq:cca-1}
- \Delta_\infty u \leq c \quad \mbox{in} \ V.
\end{equation}
Let $x_0 \in \R^n$ and $a, b, c \in \R$. If the quadratic cone $\varphi$ given by
\begin{equation}\label{eq:cca-2}
\varphi(x) := a + b |x - x_0| - \frac{c}{2} |x - x_0|^2
\end{equation}
satisfies
\begin{equation}\label{eq:cca-3}
\varphi \in C^2(V) \quad \mbox{and} \quad u \leq \varphi \ \mbox{on} \ \partial V,
\end{equation}
then
\begin{equation}
\label{eq:cca-ineq}
u \leq \varphi \quad \mbox{in}  \ V.
\end{equation}
\end{lem}

\begin{proof}
Without loss of generality, we assume that $a = 0$ and $x_0 = 0$. We argue by contradiction, and assume that \EQ{cca-3} holds but \EQ{cca-ineq} fails. We may select $\tilde c$ slightly greater than $c$ such that if we set $\tilde\varphi (x) :=  b |x| - \frac{1}{2} \tilde c |x|^2$, then the function $u - \tilde \varphi$ attains its maximum at a point $x_1\in V$ such that
\begin{equation*}
(u - \tilde \varphi) (x_1) > \sup_{\partial V} (u-\tilde \varphi).
\end{equation*}
Notice that the gradient of $\tilde \varphi$ a point $x \neq 0$ is given by
\begin{equation*}
D\tilde\varphi(x) = \left( b|x|^{-1} - \tilde c\right) x,
\end{equation*}
and for $x\neq 0$, we have
\begin{equation*}
D^2 \tilde \varphi(x) = \left( b|x|^{-1} - \tilde c\right) \iden - b |x|^{-3} x \otimes x.
\end{equation*}
Thus $D^2\tilde \varphi(x)$ has eigenvalues $-\tilde c$ with multiplicity 1 and associated eigenvector $x$, and $b|x|^{-1} - \tilde c$ with multiplicity $n-1$. We break our argument into multiple cases and deduce a contradiction in each of them.

\emph{Case 1:} $b=0$. In this case, every eigenvalue of $D^2\tilde\varphi(x_1)$ is $-\tilde c$, and hence
\begin{equation*}
-\Delta^+_\infty \tilde\varphi(x_1) = \tilde c > c, 
\end{equation*}
a contradiction.

\emph{Case 2:} $b\neq 0$ and $\tilde c \leq 0$. Since $\varphi$ is $C^2$ in $V$, we must have $0\not \in V$. In particular, $x_1 \neq 0$. The vector $D\tilde \varphi(x_1)$ lies in the direction of $x_1$, provided it is nonzero. If $D\varphi(x_1) = 0$, then the eigenvalues of $D^2\tilde \varphi(x_1)$ are $-\tilde c$ and $0 \leq -\tilde c$. In both cases we have
\begin{equation*}
-\Delta^+_\infty \tilde \varphi(x_1) = \tilde c > c,
\end{equation*}
a contradiction.

\emph{Case 3:} $b\neq 0$ and $\tilde c > 0$. As in Case 2, we have $0 \not \in V$, and in the event that $D \tilde \varphi(x_1) \neq 0$, the argument proceeds exactly as in Case 2. Otherwise, $D\tilde\varphi(x_1) = 0$, and thus $b > 0$ and $|x_1| = b / \tilde c = : r$. For small $\alpha, \beta > 0$ such that $\tilde c \beta < \alpha$, define
\begin{equation*}
\psi_{\alpha, \beta}(x) :=
\left\{ \begin{array}{ll}
(b + \alpha)|x| - \frac{1}{2}\tilde c|x|^2 & \mbox{if}\  |x| \leq r + \beta \\
(b + \alpha)(2(r + \beta) - |x|) - \frac{1}{2}\tilde c(2(r+\beta) - |x|)^2 & \mbox{if} \ |x| > r + \beta.
\end{array} \right.
\end{equation*}
Notice that $\psi_{\alpha, \beta}$ is smooth on the set $\{ |x| \neq r + \beta \}$. For $x \in \Omega$ and $|x|\neq r + \beta$, the gradient $D\psi_{\alpha,\beta}(x)$ is nonzero, lies in the direction of $x$, and is an eigenvector of the Hessian $D^2\psi_{\alpha,\beta}(x)$ with eigenvalue $-\tilde c$. Also observe, for $\beta_1 < \beta_2$, that $\psi_{\alpha, \beta_1} \equiv \psi_{\alpha, \beta_2}$ on $\{ |x | \leq r + \beta_1 \}$ and $\psi_{\alpha, \beta_1} < \psi_{\alpha, \beta_2}$ on $\{ |x | > r + \beta_1 \}$. 

By selecting $\alpha$ small enough, we may assume that $u - \psi_{\alpha, \beta}$ has a maximum at some $x_2 \in \Omega$. If $x_2$ can be selected such that $|x_2| \neq r + \beta$, then the preceding discussion yields the contradiction
\begin{equation*}
-\Delta_\infty^+ \psi_{\alpha,\beta}(x_2) = \tilde{c} > c.
\end{equation*}
If every point $x$ of local maximum of $u - \varphi_{\alpha,\beta}$ in $V$ satisfies $|x| = r + \beta$, then we may make $\beta$ slightly larger. Now the difference $u - \psi_{\alpha, \beta}$ must have a maximum in the set $\left\{ |x| < r + \beta \right\}$, completing the proof.
\end{proof}


Following Crandall \cite{Crandall:2008}, we introduce the \emph{increasing slope estimates} \EQ{slope-below} and \EQ{slope-above} for subsolutions of the infinity Laplace equation. Our proof follows the argument given in \cite{Crandall:2008} for the case $f \equiv 0$. Before stating the estimates, we introduce some notation. Let us define the \emph{local Lipschitz constant} $L(u,x)$ of a function $u: \Omega \to \R$ at a point $x\in \Omega$ by
\begin{equation*}
L(u,x) := \lim_{r \downarrow 0} \Lip\left( u , B(x,r) \right) = \inf \left\{ \Lip\left( u, B(x,r) \right) : 0 < r < \dist(x,\partial \Omega) \right\}.
\end{equation*}
It is easy to check (see \cite[Lemma 4.3]{Crandall:2008}) that $x \mapsto L(u,x)$ is upper semicontinuous,
\begin{equation}
\label{eq:local-Lipschitz}
\Lip(u, \Omega) \leq \sup_{y \in \Omega} L(u,y),
\end{equation}
and if the right side of \EQ{local-Lipschitz} is finite, then
\begin{equation}
\label{eq:local-Lipschitz-Du}
\| Du \|_{L^\infty(\Omega)} = \sup_{y\in \Omega } L(u,y).
\end{equation}

\begin{lem}
\label{lem:slope-estimates}
Suppose that $u \in USC(\bar \Omega)$ is a viscosity subsolution of
\begin{equation*}
- \Delta_\infty u \leq h \mbox{ in } \Omega,
\end{equation*}
for some $h \in USC(\Omega)$. If $r > 0$ and $\dist(x, \partial \Omega) \geq r$, then
\begin{equation}
\label{eq:slope-below}
L(u,x) \leq S^+_r u(x) + \frac{r}{2} \max_{\bar B(x,r)} h.
\end{equation}
If, in addition, $y \in \bar B(x,r)$ is such that $u(y) = \sup_{\bar B(x,r)} u$, then
\begin{equation}\label{eq:slope-above}
L(u,y) \geq S^+_r u(x) - \frac{r}{2} \max_{\bar B(x,r)} h.
\end{equation}
\end{lem}

\begin{proof}
Let $r$, $x$, and $y$ be as above, and define
\begin{equation*}
c := \max_{\bar B(x,r)} h.
\end{equation*}
The following inequality is easily seen to be valid for $z \in \partial \left( B(x,r) \setminus \{ x \} \right)$, and therefore, by \LEM{cca}, for all $z \in \bar B(x,r)$:
\begin{equation}
\label{eq:lip-bound-eq1}
u(z) - u(x) \leq \left( \frac{u(y) - u(x)}{r} + \frac{c}{2} r \right) |z - x| - \frac{c}{2} |z - x|^2.
\end{equation}
Fixing $z \in B(x,r)$ and applying \LEM{cca} again, we obtain the inequality
\begin{equation}
\label{eq:lip-bound-eq2}
u(w) - u(z) \leq \left( \frac{u(y) - u(z)}{r - |x - z|} + \frac{c}{2}(r + |x - z|) \right) |w - z| - \frac{c}{2}|w - z|^2
\end{equation}
for every $w \in \bar B(x,r)$, since it holds for every $w \in \partial \left( B(x,r) \setminus \{ z \} \right)$.

The inequality \EQ{lip-bound-eq1} implies $\limsup_{z \to x} u(z) \leq u(x)$. Substituting $w = x$ in \EQ{lip-bound-eq2} and sending $z \to x$, we obtain $\liminf_{z \to x} u(z) \geq u(x)$. If we divide \EQ{lip-bound-eq2} by $|w - z|$, we get
\begin{equation*}
\frac{u(w) - u(z)}{|w - z|} \leq \frac{u(y) - u(z)}{r - |x - z|} + \frac{c}{2}(r + |x - z|) - \frac{c}{2}|w - z|.
\end{equation*}
Sending $w, z \to x$ and using $u(z) \to u(x)$ yields \EQ{slope-below}.

Now we prove \EQ{slope-above}. If $c < 0$, then by \LEM{cca}, $u \leq \psi$ in $B(x,r)$ for
\begin{equation*}
\psi(w) := \max_{\bar B(x,r)} u - \frac{c}{2} r^2 + \frac{c}{2} |w - x|^2.
\end{equation*}
Thus we must have $|y - x| = r$ in the case that $c < 0$. In particular, if $y=x$, then we immediately conclude that $c\geq 0$ and $S^+_r u(x) = 0$, and thus \EQ{slope-above} is trivial. We may therefore assume that $y\neq x$. Define the quadratic cone
\begin{equation*}
\varphi(z) : = u(x) + \left( \frac{u(y) - u(x)}{r} + \frac{c}{2} r \right) |z - x| - \frac{c}{2} |z - x|^2.
\end{equation*}
According to \LEM{cca}, $u \leq \varphi$ in $\bar B(x,r)$. Regardless of the sign of $c$,
\begin{align*}
D \varphi(y) \cdot \frac{y - x}{|y-x|} & = \left( \frac{u(y) - u(x)}{r} + \frac{r}{2} c \right) - c|y-x| \\
& \geq \frac{u(y) - u(x)}{r} - \frac{r}{2} c \\
& = S^+_r u(x) - \frac{r}{2}c.
\end{align*}
We may assume the quantity on the right-hand side is positive, as otherwise \EQ{slope-above} is trivial. Since $u$ touches $\varphi$ from below at $y$ and $\varphi$ is increasing as we move towards $\partial B(x, |y-x|)$, we have $L(u,y) \geq | D \varphi(y) |$. Thus \EQ{slope-above} holds.
\end{proof}


The following proposition is essential to our theory for the continuum equation. It shows that subsolutions of the continuum equation are essentially finite difference subsolutions, possibly with a small error. This proposition bears a resemblance to \LEM{max-over-balls}, but as we see below, it is much stronger.

\begin{prop} \label{prop:continuous-sup-convolution}
Assume that $u \in USC(\Omega)$ is a viscosity subsolution of 
\begin{equation*}
- \Delta_\infty u \leq h \mbox{ in } \Omega,
\end{equation*}
with $h \in USC(\Omega)$. For $\ep > 0$ and $x \in \Omega_{2\ep}$, define
\begin{equation*}
u^\ep(x) := \max_{\bar B(x,\ep)} u \quad \mbox{and} \quad h^{2\ep}(x) := \max_{\bar B(x,2 \ep)} h.
\end{equation*}
Then 
\begin{equation*}
-\Delta^\ep_\infty u^\ep \leq h^{2\ep} \quad \mbox{in} \  \Omega_{2 \ep}.
\end{equation*}
\end{prop}
\begin{proof}
Fix some $x_0 \in \Omega_{2 \ep}$. Choose $y_0 \in \bar B(x_0,\ep)$ and $z_0 \in \bar B(y_0, \ep)$ such that
\begin{equation*}
u(y_0) = u^\ep(x_0) \quad \mbox{and} \quad u(z_0) = u^\ep(y_0).
\end{equation*}
Notice that
\begin{equation*}
\ep S^-_\ep u^\ep(x_0) = u^\ep(x_0)  -  \min_{y\in \bar B(x_0,\ep)} u^\ep(y) \leq u(y_0) - u(x_0) = \ep S^+_\ep u(x_0),
\end{equation*}
and similarly
\begin{equation*}
\ep S^+_\ep u^\ep(x_0) = \max_{\bar B(x_0, \ep)} u^\ep - u^\ep(x_0) \geq u(z_0) - u(y_0) = \ep S^+_\ep u(y_0).
\end{equation*}
Combining these inequalities, we see that
\begin{equation}
\label{eq:ineq1}
- \Delta_\infty^\ep u^\ep(x_0) = \frac{1}{\ep} \left( S^-_\ep u^\ep(x_0) - S^+_\ep u^\ep(x_0) \right) \leq \frac{1}{\ep} \left( S^+_\ep u(x_0) - S^+_\ep u(y_0) \right).
\end{equation}
The increasing slope estimates \EQ{slope-below} and \EQ{slope-above} imply that
\begin{equation}
\label{eq:ineq2}
S^+_\ep u(x_0) \leq L(u, y_0) + \frac{\ep}{2} \max_{\bar B(x_0, \ep)} h \leq L(u, y_0) + \frac{\ep}{2} h^{2\ep}(x_0),
\end{equation}
and
\begin{equation}
\label{eq:ineq3}
S^+_\ep u(y_0) \geq L(u, y_0) - \frac{\ep}{2} \max_{\bar B(y_0, \ep)} h \geq L(u, y_0) - \frac{\ep}{2} h^{2\ep}(x_0).
\end{equation}
Combining \EQ{ineq1}, \EQ{ineq2}, and \EQ{ineq3} we deduce that  $- \Delta_\infty^\ep u^\ep(x_0) \leq h^{2\ep}(x_0)$.
\end{proof}


As a first application of \PROP{continuous-sup-convolution}, we give a second proof of \THM{convergence} which is more efficient than the proof in the previous section. We follow the outline of the first proof, but manage to avoid the diadic sequences and the appeal to \LEM{discrete-to-continuous}. The idea is to perturb the test function rather than the sequence $\{ u_j \}$, in the spirit of Evans \cite{Evans:1989}, and then invoke \PROP{continuous-sup-convolution}.

\begin{proof}[{\bf Second proof of \THM{convergence}}] Suppose $\ep_j \downarrow 0$, $u_j \rightarrow u$ locally uniformly in $\Omega$, and each $u_j \in USC(\Omega)$ satisfies
\begin{equation*}
- \Delta_\infty^{\ep_j} u_j \leq f \mbox{ in } \Omega_{\ep_j}.
\end{equation*}
Suppose $\varphi \in C^\infty(\Omega)$ is such that $u - \varphi$ has a strict local maximum at a point $x_0 \in \Omega$. We must show that
\begin{equation}\label{eq:second-proof-convergence-wts}
- \Delta_\infty^+ \varphi(x_0) \leq f(x_0).
\end{equation}
Define
\begin{equation*}
\varphi_j(x) := \min_{\bar B(x,\ep_j)} \varphi.
\end{equation*}
Since $\varphi$ is viscosity supersolution of
\begin{equation*}
- \Delta_\infty \varphi \geq - \Delta_\infty^+ \varphi \mbox{ in } \Omega,
\end{equation*}
\PROP{continuous-sup-convolution} yields
\begin{equation}
\label{eq:new-converge-eq1}
- \Delta_\infty^{\ep_j} \varphi_j(x) \geq \min_{\bar B(x,2\ep_j)} \left( - \Delta_\infty^+ \varphi \right) \mbox{ for every } x \in \Omega_{2 \ep_j}.
\end{equation}
We may assume there exist $x_j \in \Omega_{2\ep_j}$ such that $x_j \rightarrow x_0$ and
\begin{equation*}
(u_j - \varphi_j)(x_j) = \max_{\bar \Omega_{\ep_j}} (u_j - \varphi_j).
\end{equation*}
We obtain
\begin{equation}
\label{eq:new-converge-eq2}
- \Delta_\infty^{\ep_j} \varphi_j(x_j) \leq - \Delta_\infty^{\ep_j} u_j(x_j) \leq f(x_j).
\end{equation}
Combining \EQ{new-converge-eq1} and \EQ{new-converge-eq2}, we obtain
\begin{equation*}
- \max_{\bar B(x_j,2\ep_j)} \Delta_\infty^+ \varphi \leq f(x_j).
\end{equation*}
Passing to the limit $j \to \infty$ using the upper semicontinuity of $x \mapsto \Delta_\infty^+ \varphi(x)$, we obtain \EQ{second-proof-convergence-wts}.
\end{proof}


As a second application of \PROP{continuous-sup-convolution}, we give a simple proof of Jensen's theorem on the uniqueness of infinity harmonic functions.

\begin{thm}[Jensen \cite{Jensen:1993}]\label{thm:Jensen}
Suppose that $u, -v\in USC(\bar \Omega)$ satisfy
\begin{equation}\label{eq:jensen-eq}
-\Delta_\infty u \leq 0 \leq -\Delta_\infty v \quad \mbox{in} \ \Omega.
\end{equation}
Then 
\begin{equation*}
\max_{\bar \Omega} (u-v) = \max_{\partial \Omega} (u-v).
\end{equation*}
\end{thm}
\begin{proof}
For $\ep > 0$, let $u^\ep (x):= \max_{\bar B(x,\ep)} u$ and $v_\ep (x):= \min_{\bar B(x,\ep)} v$. By \PROP{continuous-sup-convolution}, we have
\begin{equation*}
- \Delta_\infty^\ep u^\ep \leq 0 \leq - \Delta_\infty^\ep v_\ep \quad \mbox{in} \quad \Omega_{2\ep}.
\end{equation*}
By \PROP{no-extrema-comparison-ep-thick}, we have
\begin{equation*}
\sup_{\Omega_\ep} (u^\ep - v_\ep) \leq \sup_{\Omega_\ep \setminus \Omega_{2 \ep}} (u^\ep - v_\ep).
\end{equation*}
Passing to the limit $\ep \to 0$ and using the upper semicontinuity of $u$ and $-v$ yields the theorem.
\end{proof}

\begin{remark}
The only essential ingredients in our proof Jensen's theorem are Propositions \ref{prop:no-extrema-comparison-ep-thick} and \ref{prop:continuous-sup-convolution}, the latter of which depends only on \LEM{slope-estimates} and the easy part ($c=0$) of \LEM{cca}. See also the authors' recent note \cite{Armstrong:preprint} for an efficient, self-contained proof of Jensen's uniqueness result. 
\end{remark}


We need the following continuity estimates for solutions of the continuum infinity Laplace equation.

\begin{lem}
\label{lem:continuum-estimates}
Assume that $u \in C(\bar \Omega)$ is a viscosity solution of
\begin{equation*}
\left\{ \begin{array}{ll}
- \Delta_\infty u = f & \mbox{ in } \Omega \\
u = g & \mbox{ on } \partial \Omega.
\end{array} \right.
\end{equation*}
Then $u$ is locally Lipschitz in $\Omega$, and there exists a modulus $\omega: [0,\infty) \to [0,\infty)$ such that $\omega(0) = 0$ and
\begin{equation}
\label{eq:continuum-modulus}
|u(x) - u(y)| \leq \omega(|x - y|) \quad \mbox{for all} \ \ x,y\in \bar\Omega. 
\end{equation}
The modulus $\omega$ depends only on $\diam(\Omega)$, the modulus $\omega_g$, and $\| f \|_{L^\infty(\Omega)}$. Moreover, there is a constant $C$ depending only on $\diam(\Omega)$ such that
\begin{equation}
\label{eq:bcont-cont}
|u(x) - g(x_0)| \leq 2 \omega_g(|x - x_0|) + C \| f \|_{L^\infty(\Omega)} |x - x_0| \quad \mbox{for all} \ \ x \in \Omega, \ x_0 \in \partial \Omega.
\end{equation}
\end{lem}

\begin{proof}
Fix $x_0 \in \partial \Omega$, $\delta > 0$, and $c \geq \| f \|_{L^\infty(\Omega)}$. Define
\begin{equation*}
\varphi(x) : = g(x_0) + \omega_g(\delta) + \left(\frac{\omega_g(\delta)}{\delta} + \frac{c}{2} \diam(\Omega) \right) |x - x_0| - \frac{c}{2} |x - x_0|^2.
\end{equation*}
According to \LEM{cca}, we have $u \leq \varphi$ in $\Omega$. Substituting $\delta := |x - x_0|$ immediately produces \EQ{bcont-cont}.

The inequality \EQ{bcont-cont} immediately gives the estimate
\begin{equation}
\label{eq:cont-est-sup}
\| u \|_{L^\infty(\bar \Omega)} \leq C \left( \|g\|_{L^\infty(\partial \Omega)} + \| f \|_{L^\infty(\Omega)} \right),
\end{equation}
for some constant $C$ depending only on $\diam(\Omega)$. Since $u$ is bounded, the increasing slope estimate \EQ{slope-below} combined with \EQ{local-Lipschitz} and \EQ{local-Lipschitz-Du} imply that $u$ is locally Lipschitz, and
\begin{equation}
\label{eq:cont-est-lip}
\| Du \|_{L^\infty(\Omega_r)} \leq \frac{2}{r} \| u \|_{L^\infty(\Omega_r)} + \frac{r}{2} \| f \|_{L^\infty(\Omega_r)},
\end{equation}
for each $r > 0$. Combining \EQ{cont-est-sup} and \EQ{cont-est-lip} yields
\begin{equation}
\label{eq:cont-est-w1inf}
\| u \|_{W^{1,\infty}(\Omega_r)} \leq C (1 + r^{-1})( \| g \|_{L^\infty(\partial \Omega)} + \| f \|_{L^\infty(\Omega)} ),
\end{equation}
for every $r > 0$. It is now easy to combine \EQ{bcont-cont} and \EQ{cont-est-w1inf} to obtain \EQ{continuum-modulus}.
\end{proof}

\begin{remark}
The above Lemma contains the continuum analogues of \LEM{bcont} and \LEM{epcont}, but does not contain the analogue of \LEM{icont}. We could obtain this by studying the gradient flowlines, as in \cite[Section 6]{Crandall:2008}, using the increasing slope estimate \EQ{slope-above}. It can also be obtained by passing to the limit $\ep \to 0$ in the estimates \EQ{icont-bound}, in the case that we have uniqueness for the continuum equation.
\end{remark}


We now obtain \THM{stability} from \PROP{continuous-sup-convolution} and \THM{convergence}.

\begin{proof}[{\bf Proof of \THM{stability}}]
By \EQ{continuum-modulus}, we may assume $u_k \rightarrow u \in C(\bar \Omega)$ uniformly. By symmetry, it remains to check that $u$ is a viscosity subsolution of $- \Delta_\infty u \leq f$ in $\Omega$. Let $\Omega_1 \subset \subset \Omega$ be given. Set $\ep_k := 1/k$ and define
\begin{equation*}
\tilde{f}_k (x) : = \max_{\bar B(x,2\ep_k)} f_k \quad \mbox{and} \quad \tilde{u}_k (x) : = \max_{\bar B(x,\ep_k)} u_k
\end{equation*}
It is clear that $\tilde f_k \rightarrow f$ and $\tilde{u}_k \rightarrow u$ uniformly on $\bar \Omega_1$ as $k\to \infty$. According to \PROP{continuous-sup-convolution},
\begin{equation*}
-\Delta_\infty^{\ep_k} \tilde u_k  \leq \tilde f_k \quad \mbox{in} \ \Omega_{2\ep_k}.
\end{equation*}
Thus for given $\eta > 0$, we have
\begin{equation*}
-\Delta_\infty^{\ep_k} \tilde u_k  \leq f + \eta  \quad \mbox{in} \ \Omega_1
\end{equation*}
for all sufficiently large $k$. Using \THM{convergence}, we pass to the limit $k\to \infty$ to deduce that
\begin{equation*}
-\Delta_\infty u \leq f + \eta \quad \mbox{in} \ \Omega_1.
\end{equation*}
We may send $\eta \to 0$ to get
\begin{equation*}
-\Delta_\infty u \leq f  \quad \mbox{in} \ \Omega_1.
\end{equation*}
The above inequality holds for any subdomain $\Omega_1 \subset \subset \Omega$, hence it holds in the whole domain $\Omega$.
\end{proof}


The next proposition is well-known (see, e.g., \cite[Theorem 3.1]{Lu:2008}). Our proof, which is very similar to the proof of \THM{Jensen} above, is based on \PROP{continuous-sup-convolution} and does not invoke deep viscosity solution machinery.

\begin{prop}\label{prop:strict-comparison}
Assume that $f, \tilde f \in C(\Omega)$ are such that $f < \tilde f$, and that $u, - v \in USC(\bar \Omega)$ satisfy
\begin{equation*}
-\Delta_\infty u \leq f \quad \mbox{and} \quad -\Delta_\infty v \geq \tilde f \quad \mbox{in} \ \Omega,
\end{equation*}
in the viscosity sense. Then
\begin{equation*}
\max_{\bar \Omega} (u - v) = \max_{\partial \Omega} (u - v).
\end{equation*}
\end{prop}

\begin{proof}
For $\ep > 0$ and $x \in \Omega_\ep$, let $u^\ep(x) := \max_{\bar B(x,\ep)} u$ and $v_\ep(x) := \min_{\bar B(x,\ep)} v$. For $\ep > 0$ and $x \in \Omega_{2 \ep}$, let $f^{2\ep}(x) := \max_{\bar B(x,2\ep)} f$ and $\tilde f_{2\ep}(x) := \min_{\bar B(x,2\ep)} \tilde f$.

For $r > 0$, select $\ep = \ep(r)$ such that $0 < \ep < r$ and $f^{2\ep} < \tilde f_{2\ep}$ in $\Omega_{r + \ep}$.
By \PROP{continuous-sup-convolution},
\begin{equation*}
- \Delta_\infty^\ep u^\ep \leq f^{2\ep} < \tilde f_{2\ep} \leq - \Delta_\infty^\ep v_\ep \quad \mbox{in} \quad \Omega_{r + \ep}.
\end{equation*}
According to \REM{strict-comparison-discrete},
\begin{equation*}
\sup_{\Omega_r} (u^\ep - v_\ep) \leq \sup_{\Omega_r \setminus \Omega_{r + \ep}} (u^\ep - v_\ep).
\end{equation*}
Using the upper semicontinuity of $u$ and $-v$, we may send $r \to 0$ to obtain the result.
\end{proof}

\begin{remark}
Yu \cite{Yu:2009} has improved \PROP{strict-comparison} by showing that, under the same hypotheses, $u \leq v$ and $u \not\equiv v$. It follows immediately that the set $\{ u < v \}$ is dense in $\Omega$. Moreover, by modifying the argument in \cite{Yu:2009}, it can be shown that $u(x) < v(x)$ whenever $L(u, x) > 0$ or $L(v, x) > 0$. This argument seems to breaks down for points $x$ for which $L(u,x) = L(v,x) = 0$ due to the singularity of the infinity Laplacian. The question of whether we have $u < v$ in $\Omega$, in general, under the hypotheses of \PROP{strict-comparison}, appears to be open.
\end{remark}

Using \THM{stability} and \PROP{strict-comparison}, we now prove \THM{existence-continuous}.

\begin{proof}[{\bf Proof of \THM{existence-continuous}}]
For each $k \geq 1$, let $u_k \in C(\bar\Omega)$ be a solution of the equation
\begin{equation*}
\left\{   \begin{aligned}
&-\Delta_\infty u_k = f + \frac{1}{k} & & \mbox{in} \ \Omega, \\
& u_k = g & & \mbox{on} \ \partial \Omega,
\end{aligned} \right.
\end{equation*}
given by \COR{existence-continuous}. By applying the estimate in \LEM{continuum-estimates} and passing to a subsequence, we may assume that there exists a function $u \in C(\bar\Omega)$ such that $u_k \rightarrow u$ uniformly on $\bar \Omega$ as $k \to \infty$. According to \THM{stability}, the function $u$ is a solution of \EQ{existence-continuous}. According to \PROP{strict-comparison}, any subsolution $v$ of \EQ{existence-continuous} with $v\leq g$ on $\partial \Omega$ satisfies $u_k \geq v$ and hence $u \geq v$ in $\Omega$. That is, $u=\bar u$ is the maximal solution of \EQ{existence-continuous}. In a similar way, we can argue that \EQ{existence-continuous} possesses a minimal solution $\underline u$.
\end{proof}

\begin{proof}[{\bf Proof of \THM{uniqueness-generic}}]
For each $c \in \R$, let $\bar u_c$ and $\underline u_c$ denote the maximal and minimal solutions of the problem
\begin{equation}\label{eq:proof-uniqueness-generic}
\left\{ \begin{aligned}
& -\Delta_\infty u = f + c & & \mbox{in} \ \Omega, \\
& u = g & & \mbox{on} \ \partial \Omega,
\end{aligned} \right.
\end{equation}
which exist by \THM{existence-continuous}. Let $\mathcal{N} : = \left\{ c \in \R: \bar u_c \not\equiv \underline u_c \right\}$ be the set of points for which the problem \EQ{proof-uniqueness-generic} has two or more solutions. Define
\begin{equation*}
I^+(c):= \int_\Omega \bar u_c \, dx \quad \mbox{and} \quad I^-(c) : = \int_\Omega \underline u_c\, dx.
\end{equation*}
It is clear that $\mathcal{N} = \left\{ c : I^+(c) \neq I^-(c) \right\}$. According to \PROP{strict-comparison}, we have $I^+(c_1) \leq I^-(c_2) \leq I^+(c_2)$ for all constants $c_1 < c_2$. It follows that $I^+(c_1) = I^-(c_1)$ at any point $c_1\in \R$ for which
\begin{equation*}
I^+(c_1) = \lim_{c \downarrow c_1} I^+(c).
\end{equation*}
Since $I^+$ is increasing, it can have only countably many points of discontinuity. Thus $\mathcal{N}$ is at most countable.
\end{proof}

\begin{proof}[{\bf Proof of \THM{uniqueness-generic-discrete}}]
The argument is nearly identical to the proof of \THM{uniqueness-generic}. We simply replace the mentions of \THM{existence-continuous} and \PROP{strict-comparison} by \THM{existence} and \REM{strict-comparison-discrete}, respectively.
\end{proof}


\section{Uniqueness, continuous dependence, and rates of convergence} \label{sec:six}

In this section, we study the relationship between uniqueness for the continuum infinity Laplace equation and continuous dependence of solutions of the finite difference equation. As a corollary to this study, we obtain explicit estimates of the rate at which solutions of the finite difference equation converge to solutions of the continuum equation.

Adapting the proof of \THM{Jensen}, we show that continuous dependence at $f$ for the finite difference equation implies uniqueness for the continuum equation with $f$ on the right-hand side.

\begin{prop} \label{prop:dep-uniq}
Suppose that $H_f : [ 0, \infty) \to [0,\infty)$ is a continuous function with $H_f(0)=0$, and suppose $H_f$ has the property that
\begin{equation*}
\sup_{\Omega_\ep} (\tilde u - \tilde v) \leq \sup_{\Omega_\ep \setminus \Omega_{2\ep}}(\tilde u - \tilde v) + H_f(\ep),
\end{equation*}
whenever $\ep > 0$ and $\tilde u, - \tilde v \in USC(\Omega_\ep)$ and satisfy $0 < \tilde u, \tilde v < 1$ and
\begin{equation}
\label{eq:dep-uniq-subsup}
- \Delta_\infty^\ep \tilde u \leq f^{2\ep} \quad \mbox{and} \quad - \Delta_\infty^\ep \tilde v \geq f_{2\ep} \quad \mbox{in} \ \Omega_{2 \ep},
\end{equation}
where we have set
\begin{equation*}
f^{2\ep}(x) := \max_{\bar B(x,2\ep)} f \quad \mbox{and} \quad f_{2\ep}(x) := \min_{\bar B(x,2\ep)} f.
\end{equation*}
If $u,-v\in \USC(\bar\Omega)$ satisfy $0 < u, v < 1$ and
\begin{equation}\label{eq:prop-dep-uniq-1}
- \Delta_\infty u \leq f \leq -\Delta_\infty v \quad \mbox{in} \ \Omega, \\
\end{equation}
then
\begin{equation*}
\max_{\bar \Omega} (u - v) = \max_{\partial \Omega} (u - v).
\end{equation*}
\end{prop}

\begin{proof}
Suppose $u, - v \in USC(\bar \Omega)$ satisfy the hypotheses above. Setting
\begin{equation*}
u^\ep(x) := \max_{\bar B(x,\ep)} u, \quad \mbox{and} \quad v_\ep(x) := \min_{\bar B(x,\ep)} v,
\end{equation*}
we see from \PROP{continuous-sup-convolution} that \EQ{dep-uniq-subsup} holds. By hypothesis, we have
\begin{equation}\label{eq:prop-dep-uniq-2}
\sup_{\Omega} (u - v) \leq \sup_{\Omega_\ep} (u^\ep - v_\ep) \leq \sup_{\Omega_\ep \setminus \Omega_{2\ep}}(u^\ep - v_\ep) + H_f(\ep).
\end{equation}
Since $u - v$ is upper semicontinuous and $H_f(0) = 0$, sending $\ep \to 0$ yields the proposition.
\end{proof}

\begin{remark}
A compactness argument combined with \THM{convergence} easily yields a converse to \PROP{dep-uniq}.
\end{remark}


We now consider a simple continuous dependence result.

\begin{prop} \label{prop:dep-pos}
Assume that $f \geq a > 0$ for some positive $a>0$, and $\ep, \gamma > 0$. Suppose $u, -v \in USC(\Omega)$ satisfy
\begin{equation*}
- \Delta_\infty^\ep u \leq f + \gamma \quad \mbox{and} \quad - \Delta_\infty^\ep v \geq f \quad \mbox{in} \ \Omega_\ep.
\end{equation*}
Then 
\begin{equation}
\label{eq:dep-pos}
\sup_{\Omega} (u - v) \leq \sup_{\Omega \setminus \Omega_\ep} (u - v) + \| v \|_{L^\infty(\Omega)} a^{-1} \gamma.
\end{equation}
\end{prop}
\begin{proof}
Define $w:= (1 + a^{-1} \gamma) v$ and notice that
\begin{equation*}
- \Delta_\infty^\ep w \geq f + \gamma.
\end{equation*}
According to \PROP{no-extrema-comparison-ep-thick},
\begin{equation*}
\sup_{\Omega} (u-w) \leq \sup_{\Omega \setminus \Omega_\ep} (u-w).
\end{equation*}
From this we obtain
\begin{align*}
\sup_{\Omega} (u-v) & \leq \sup_{\Omega \setminus \Omega_\ep} (u-v) + \sup_{\Omega} (w-v) \\
& = \sup_{\Omega \setminus \Omega_\ep} (u-v) + a^{-1} \gamma \| v \|_{L^\infty(\Omega)}. \qedhere
\end{align*}
\end{proof}

Armed with the continuous dependence estimate \EQ{dep-pos} and \PROP{dep-uniq}, we now prove \THM{comparison-continuous}.

\begin{proof}[{\bf Proof of \THM{comparison-continuous}}]
The theorem in the case that $f\equiv 0$ is \THM{Jensen}, which we proved in \SEC{five}. We only argue in the case $f > 0$, since the conclusion in the case that $f< 0$ obviously follows by symmetry. We may further assume that $f$ is uniformly continuous and bounded below by a positive constant $a > 0$. Indeed, if the conclusion fails for some functions $u,v\in C(\bar\Omega)$, the conclusion will still fail in a slightly smaller domain $\tilde \Omega \subset\subset \Omega$ after we subtract a very small positive constant from $u$. The theorem is now immediate from Propositions \ref{prop:dep-uniq} and \ref{prop:dep-pos}.
\end{proof}

The above argument yields more information when we have control over the modulus of continuity for $f$ and $g$, in which case we obtain the following rate of convergence. In the case $\alpha = 1$, the analogue of this result for standard tug-of-war appeared in \cite[Theorem 1.3]{Peres:2009}.

\begin{prop} \label{prop:rate-pos}
Assume that $0 < \alpha \leq 1$, $g \in C^{0,\alpha}(\partial \Omega)$, $f \in C^{0, \alpha}(\Omega)$, and $|f| \geq a$ for some positive constant $a > 0$. Let $u_\ep \in C(\bar \Omega)$ be the unique solution of the problem
\begin{equation*}
\left\{ \begin{aligned}
& - \Delta_\infty^\ep u_\ep = f & \mbox{in} && \Omega, \\
& u_\ep = g & \mbox{on} && \partial \Omega,
\end{aligned} \right.
\end{equation*}
and $u \in C(\bar \Omega)$ denote the unique solution of
\begin{equation*}
\left\{ \begin{aligned}
& - \Delta_\infty u = f & \mbox{in} && \Omega, \\
& u = g & \mbox{on} && \partial \Omega.
\end{aligned} \right.
\end{equation*}
Then there exists a constant $C>0$, depending only $\diam(\Omega)$, $\| f \|_{C^{0,\alpha}(\Omega)}$, and $\| g \|_{C^{0,\alpha}(\partial \Omega)}$ such that
\begin{equation*}
\| u - u_\ep \|_{L^\infty(\bar \Omega)} \leq C \left(1 + a^{-1} \right)\ep^\alpha.
\end{equation*}
\end{prop}

\begin{proof}
Set
\begin{equation*}
u^\ep(x) = \max_{\bar B(x,\ep)} u \quad \mbox{and} \quad f^{2\ep}(x) = \max_{\bar B(x,2 \ep)} f.
\end{equation*}
\PROP{dep-pos} gives
\begin{equation*}
\sup_{\Omega} (u - u_\ep) \leq \sup_{\Omega_\ep} (u^\ep - u_\ep) \leq \sup_{\Omega_\ep \setminus \Omega_{2 \ep}} (u^\ep - u_\ep) + C a^{-1} \sup_{\Omega_{2\ep}} (f^{2\ep} - f).
\end{equation*}
We estimate
\begin{equation*}
\sup_{\Omega_{2\ep}} (f^{2\ep} - f) \leq \left[ f \right]_{C^{0,\alpha}(\Omega)} \cdot(2\ep)^\alpha.
\end{equation*}
Since $g \in C^{0,\alpha}(\Omega)$, the boundary continuity estimates \EQ{bcont-rhobound} and \EQ{bcont-cont} yield
\begin{equation*}
\sup_{\Omega_\ep \setminus \Omega_{2 \ep}} (u^\ep - u_\ep) \leq C\ep^\alpha,
\end{equation*}
where $C$ depends on the appropriate constants. Combining these two estimates, we obtain $\sup_{\Omega} (u - u_\ep) \leq C\ep^\alpha$. A similar argument produces $\sup_{\Omega} (u_\ep - u) \leq C\ep^\alpha$, which completes the proof.
\end{proof}


We now proceed to prove \THM{dependence-zero}, the continuous dependence result for $f \equiv 0$. Before we give the proof, we need to construct a strictness transformation with explicit constants.

\begin{lem}
\label{lem:super-strict}
Assume $v \in C(\Omega)$ satisfies $0 < v < 1$, and
\begin{equation*}
- \Delta_\infty^\ep v \geq f \geq 0 \mbox{ in } \Omega_\ep.
\end{equation*}
If $\gamma \in (0, \frac{1}{2})$ and
\begin{equation*}
w(x) := (1 + 2 \gamma) v(x) - \frac{\gamma}{2} v(x) ^2,
\end{equation*}
then $\|w - v\|_{L^\infty(\bar \Omega)} \leq \frac{3}{2} \gamma$ and
\begin{equation}
\label{eq:strict-subsol}
- \Delta_\infty^\ep w \geq f + \gamma (S^-_\ep v)^2 \quad \mbox{in} \ \Omega_\ep.
\end{equation}
\end{lem}

\begin{proof}
Suppose $v$, $\gamma$, and $w$ are as above and write $w = \lambda(v)$, where
\begin{equation*}
\lambda(z) := (1 + 2 \gamma) z - \frac{\gamma}{2} z^2.
\end{equation*}
Notice that
\begin{equation*}
\lambda'(z) > 1 \mbox{ and } \lambda''(z) = - \gamma \quad \mbox{for every} \ z \in (-1,2),
\end{equation*}
and that
\begin{equation*}
2 \lambda(a) - \lambda(a+b) - \lambda(a-b) = \gamma b^2 \mbox{ for every } a, b \in \R.
\end{equation*}
Fix $x \in \Omega_\ep$ and choose $x^+, x^- \in \bar B(x, \ep)$ such that
\begin{equation*}
- \ep^2 \Delta_\infty^\ep v(x) = 2 v(x) - v(x^-) - v(x^+) \geq \ep^2 f(x).
\end{equation*}
Notice that $\ep^2 f(x) < 1$. Since $\lambda$ is increasing, we have
\begin{equation*}
- \ep^2 \Delta_\infty^\ep w(x) = 2 \lambda(v(x)) - \lambda(v(x^-)) - \lambda(v(x^+)).
\end{equation*}
Setting $h := v(x) - v(x^-)$, we compute
\begin{align*}
- \ep^2 \Delta_\infty^\ep w(x) & = 2 \lambda(v(x)) - \lambda(v(x)-h) - \lambda(v(x)+h) + \lambda(v(x)+h) - \lambda(v(x^+)) \\
& = \gamma h^2  + \lambda(v(x)+h) - \lambda(v(x^+)) \\
& = \gamma h^2  + \lambda(2 v(x) - v(x^-)) - \lambda(v(x^+)) \\
& \geq \gamma h^2  + \lambda(v(x^+) + \ep^2 f(x)) - \lambda(v(x^+)) \\
& \geq \gamma h^2  + \ep^2 f(x) \\
& = \gamma (v(x) - v(x^-))^2  + \ep^2 f(x)
\end{align*}
Dividing both sides by $\ep^2$, we obtain \EQ{strict-subsol} at $x$.
\end{proof}


We remark that the proof below refers to a lemma of Crandall, Gunnarsson, and Wang \cite{Crandall:2007}. This is the only place where our presentation fails to be self-contained.

\begin{proof}[{\bf Proof of \THM{dependence-zero}}]
We assume first that $0 < u_\gamma < 1$, for all $0 \leq \gamma< 1/8$. Fix $0 < \gamma < 1/8$, and set
\begin{equation*}
A: = \gamma^{-\frac{1}{3}}\max_{\bar \Omega} \left( u_\gamma - u_0\right).
\end{equation*}
By inspecting\footnote{To get the fourth property listed in \EQ{patching-function}, see equation (2.18) on page 1596 of \cite{Crandall:2007} and notice their definition of $w_\ep$ a few lines above.} the proof of the patching lemma \cite[Theorem 2.1]{Crandall:2007}, we see that there exists a function $v \in C(\bar\Omega) \cap W^{1,\infty}_{\mathrm{loc}}(\Omega)$ with the following properties:
\begin{equation}\label{eq:patching-function}
\left\{ \begin{aligned}
& - \Delta_\infty v \geq 0 & & \mbox{in} \ \Omega, \\
& v = g & & \mbox{on} \ \partial \Omega, \\
& L(v,\cdot) \geq \delta & & \mbox{in} \ \Omega, \\
& |v - u_0| \leq \delta \diam(\Omega) & & \mbox{in} \ \Omega,
\end{aligned} \right.
\end{equation}
where we have set $\delta: = \frac{1}{2} A \diam(\Omega)^{-1} \gamma^{1/3}$. Select $\ep > 0$ so small that
\begin{equation*}
\max_{\bar \Omega} \left( u^\ep_\gamma - v_\ep\right) \geq \max_{\bar\Omega \setminus \Omega_{2\ep} } \left( u^\ep_\gamma - v_\ep \right) + \frac{1}{2}A \gamma^{\frac{1}{3}},
\end{equation*}
where, as usual, we have defined
\begin{equation*}
u^\ep_\gamma (x) := \max_{\bar B(x,\ep)} u_\gamma, \quad \mbox{and} \quad v_\ep(x) := \min_{\bar B(x,\ep)} v.
\end{equation*}
According to \PROP{continuous-sup-convolution}, we have
\begin{equation}
\label{eq:dep-zero-eq1}
- \Delta_\infty^\ep u^\ep_\gamma  \leq \gamma \quad \mbox{and} \quad - \Delta_\infty^\ep v_\ep \geq 0 \quad \mbox{in} \ \Omega_{2 \ep}.
\end{equation}
Now we set
\begin{equation*}
w_\ep(x) := (1 + 2 \gamma^{1/3}) v_\ep(x) - \frac{1}{2} \gamma^{1/3}(v_\ep(x))^2.
\end{equation*}
Select $x\in \Omega_{2\ep}$ and $z\in \bar B(x,\ep)$ such that $v_\ep(x) = v(z)$. Using the increasing slope estimate \EQ{slope-below}, we have
\begin{equation*}
S^-_\ep v_\ep (x) = \frac{1}{\ep} \left( v(z) -  \min_{y\in \bar B(x,2 \ep)} v(y) \right) \geq S^-_\ep v(z) \geq S^-_\ep v(x) \geq L(v,x) \geq \delta.
\end{equation*}
Notice that $\gamma^{1/3} < 1/2$ and apply \LEM{super-strict} to get
\begin{equation*}
- \Delta_\infty^\ep w_\ep \geq \gamma^{1/3} \delta^2  = \frac{1}{4} \gamma A^2 \diam(\Omega)^{-2} \quad \mbox{in} \ \Omega_{2 \ep},
\end{equation*}
as well as
\begin{equation*}
\left\| w_\ep - v_\ep \right\|_{L^\infty(\Omega)} \leq \frac{3}{2} \gamma^{1/3}.
\end{equation*}
Suppose that $A > \max\{ 6 , 2\diam(\Omega) \}$. Then 
\begin{align*}
\sup_{\Omega} \left( u^\ep_\gamma - w_\ep\right) & \geq \sup_{\Omega} \left( u^\ep_\gamma - v_\ep \right) - \frac{3}{2} \gamma^{1/3} \\
& \geq \sup_{\Omega\setminus \Omega_{2\ep} } \left( u^\ep_\gamma - v_\ep \right) + \frac{1}{2} A \gamma^{1/3}  - \frac{3}{2} \gamma^{1/3} \\
& \geq  \sup_{\Omega\setminus \Omega_{2\ep} } \left( u^\ep_\gamma - w_\ep \right) + \frac{1}{2} A \gamma^{1/3}  - 3 \gamma^{1/3} \\
& > \sup_{\Omega\setminus \Omega_{2\ep} } \left( u^\ep_\gamma - w_\ep \right).
\end{align*}
Since we have
\begin{equation*}
-\Delta^\ep_\infty u^\ep_\gamma \leq \gamma \leq \frac{1}{4} \gamma A^2 \diam(\Omega)^{-2} \leq -\Delta^\ep_\infty w_\ep,
\end{equation*}
we deduce a contradiction to \PROP{no-extrema-comparison-ep-thick}. It follows that
\begin{equation*}
A \leq \max\{ 6, 2\diam(\Omega) \}.
\end{equation*}
This completes the proof of the theorem in the case that $0 < u_\gamma < 1$. The general case may be reduced to this special case, using the estimate \EQ{cont-est-sup}.
\end{proof}

\begin{remark}
Using \THM{dependence-zero} together with \PROP{continuous-sup-convolution} and the boundary estimate \EQ{bcont-rhobound}, we deduce that
\begin{equation*}
\| u_\gamma^\ep - u_0^\ep \|_{L^\infty(\bar \Omega)} \leq C(\ep^\alpha + \gamma^{1/3}),
\end{equation*}
for $g\in C^{0,\alpha}(\Omega)$ and functions $u_\gamma^\ep \in C(\bar \Omega)$ satisfying $- \Delta_\infty^\ep u_\gamma = \gamma$ in $\Omega$ and $u = g$ on $\partial \Omega$. The dependence on $\ep$ can be removed by using a finite difference version of the patching lemma.\end{remark}


As in the case $|f| \geq a > 0$, we can combine \THM{dependence-zero} and \PROP{dep-uniq} to give another proof of \THM{Jensen}. This argument is far more complicated than the elementary proof given in \SEC{five} above. We can also use \PROP{dep-uniq} to prove a rate of convergence result for $f \equiv 0$. However, we obtain a better result by ignoring the continuous dependence estimate \EQ{dep-zero} and using \PROP{continuous-sup-convolution} directly. In the case $\alpha = 1$, the analogue of the following proposition for standard tug-of-war was obtained in \cite[Theorem 1.3]{Peres:2009}.

\begin{prop}
\label{prop:rate-zero}
Assume that $0 < \alpha \leq 1$, $g \in C^{0,\alpha}(\partial \Omega)$, and let $u_\ep \in C(\bar \Omega)$ be the unique solution of the problem
\begin{equation*}
\left\{ \begin{aligned}
& - \Delta_\infty^\ep u_\ep = 0 & \mbox{in} && \Omega, \\
& u_\ep = g & \mbox{on} && \partial \Omega,
\end{aligned} \right.
\end{equation*}
and $u \in C(\bar \Omega)$ the unique viscosity solution of
\begin{equation*}
\left\{ \begin{aligned}
& - \Delta_\infty u = 0 & \mbox{in} && \Omega, \\
& u = g & \mbox{on} && \partial \Omega.
\end{aligned} \right.
\end{equation*}
Then there exists a constant $C>0$, depending only $\diam(\Omega)$ and $\| g \|_{C^{0,\alpha}(\partial \Omega)}$, such that
\begin{equation*}
\| u - u_\ep \|_{L^\infty(\bar \Omega)} \leq C \ep^\alpha.
\end{equation*}
\end{prop}
\begin{proof}
The proof is very similar to the proof of \PROP{rate-pos}, above. The only difference is that the hypothesis that $|f| > a$ is not needed, since for $f\equiv 0$ we also have $f^{2\ep}  \equiv 0$. 
\end{proof}


\section{Open Problems}\label{sec:seven}

Let us discuss some open problems which the authors find interesting. We begin by repeating a question posed in \cite{Peres:2009}.

\begin{openproblem}\label{oprob:nonnegative-uniqueness}
If $f \geq 0$ or $f\leq 0$ in $\Omega$, does the Dirichlet problem
\begin{equation*}
\left\{ \begin{aligned}
& - \Delta_\infty u = f & \mbox{in} && \Omega, \\
& u = g & \mbox{on} && \partial \Omega,
\end{aligned} \right.
\end{equation*}
have a unique solution?
\end{openproblem}

The authors believe the answer to \OPROB{nonnegative-uniqueness} is \emph{Yes}. Motivated by  \THM{no-extrema-comparison}, we offer the following conjecture.

\begin{conj}\label{conj:no-extrema-comparison}
Assume that the functions $u, -v \in \USC(\bar\Omega)$ and $f \in C(\Omega)$ satisfy
\begin{equation*}
- \Delta_\infty u \leq f \leq - \Delta_\infty v \quad \mbox{in} \  \Omega.
\end{equation*}
Suppose also that $u$ has no strict local maximum, or $v$ has no strict local minimum, in $\Omega$. Then
\begin{equation*}
\max_{\bar \Omega} (u-v) = \max_{\partial \Omega} (u-v).
\end{equation*}
\end{conj}

Notice that if \CONJ{no-extrema-comparison} is true, then the answer to \OPROB{nonnegative-uniqueness} is \emph{Yes}. Why are we unable to prove \CONJ{no-extrema-comparison}? As explained in \SEC{six}, this question is related to the study of the dependence of the solutions $\bar u_\ep$ and $\underline u_\ep$ to the finite difference equation on the function $f$ and the parameter $\ep > 0$. If we could understand this dependence, we believe that we could prove \CONJ{no-extrema-comparison}. 

Another related question is the following.

\begin{openproblem}\label{oprob:convergence-max-min}
In the case that the continuum problem does not have uniqueness, is it true that $\underline u_\ep \rightarrow \underline u$ and $\bar u_\ep \rightarrow \bar u$ as $\ep \to 0$?
\end{openproblem}

Notice that if \OPROB{convergence-max-min} could be resolved in the affirmative, then it would follow that for any $f$ for which we have uniqueness for the finite difference problem for all $\ep > 0$, we also have uniqueness for the continuum problem. Recalling \COR{comparison}, this would imply that the answer to \OPROB{nonnegative-uniqueness} is \emph{Yes}, and we would also deduce a continuum analogue of \THM{small-f-uniqueness}.

\begin{openproblem}
Is the dependence of the right-side of the estimate \EQ{dep-zero} on the parameter $\gamma$ optimal?
\end{openproblem}

\begin{openproblem}
\LEM{epcont} provides a continuity estimate for solutions of the finite difference equation on ``scales larger than $\ep$." This can be improved in the case that $f \equiv 0$, in which we have a uniform continuity estimate on ``all scales" as shown in \PROP{minlip}. Is the hypothesis $f \equiv 0$ necessary, or should we expect a nonzero running payoff function to create oscillations in the value functions on small scales which cannot be controlled?
\end{openproblem}

\subsection*{\it{Acknowledgements}}
We thank Tonci Antunovic, Michael G. Crandall, Vesa Julin, Scott Sheffield, and Stephanie Somersille for helpful comments and suggestions.


\bibliographystyle{amsplain}
\bibliography{tugofwars}
\end{document}